\documentclass[a4paper,12pt]{article}
\usepackage{amssymb,graphicx,color,psfrag,amsmath,latexsym,amsfonts,makeidx,epsfig,float,calc,appendix}

\usepackage[scale={0.8,0.8}]{geometry}
\usepackage{amsthm}
\newtheorem{thm}{Theorem}
\newtheorem{proposition}[thm]{Proposition}

\newtheorem{lemma}[thm]{Lemma}

\newtheorem{experiment}[thm]{Experiment}
\newcommand{\blot}[1]{}

\numberwithin{equation}{section}
\numberwithin{figure}{section}
\numberwithin{thm}{section}

\begin{document}
\title{\textbf{Steady State Solutions of a Mass-Conserving Bistable Equation with a Saturating Flux}}

\author{M. Burns\footnotemark[2]\, and M. Grinfeld\footnotemark[2]}
\date{\today}
\maketitle


\renewcommand{\thefootnote}{\fnsymbol{footnote}}

\footnotetext[2]{Department of Mathematics and
Statistics, University of Strathclyde, 26 Richmond Street, Glasgow G1
1XH, Scotland, UK}

\renewcommand{\thefootnote}{\arabic{footnote}}

\maketitle
\begin{abstract} 
\noindent We consider a mass-conserving bistable equation with a saturating flux on an interval. This is the quasilinear analogue of the Rubinstein-Steinberg equation, suitable for description of order parameter conserving solid-solid phase transitions in the case of large spatial gradients in the order parameter. We discuss stationary solutions and investigate the change in bifurcation diagrams as the mass constraint and the length of the interval are varied.
\end{abstract}



\thispagestyle{plain}

\section{Introduction}\label{sect:intro} 

Consider a bounded interval, $\Omega=(0,L)\subset\mathbb{R}$, $L>0$. To extend the Ginzburg-Landau solid-solid phase transition theory to high gradients in the order parameter $u$, Rosenau \cite{ROSE} proposed the following free energy functional: \begin{equation}\label{eq:energy}
 E[u](t)=\int_\Omega \left[W(u)+\epsilon \Psi(u_x)\right] \, dx,
\end{equation}
where the diffusion coefficient $\epsilon \in(0,\infty)$, the interface energy $\Psi(s)$ is a convex function of its variable that grows linearly in $s$ so that, for example, 
$$
\Psi(s)=\sqrt{1+s^2}-1, 
$$
and $W(u)$ is the bulk energy which we take to be the double-well function,
$$
W(u)=\frac{u^4}{4}-\frac{u^2}{2}.
$$ 
The $L^2$-gradient flow of~(\ref{eq:energy}) gives the bistable Rosenau equation 
\begin{equation}\label{eq:Ros}
u_t=\epsilon\left(\psi(u_x)\right)_x+f(u), 
\end{equation}
where $f(u)=-W'(u)$ and
$$
\psi(s)=\Psi'(s)=\frac{s}{\sqrt{1+s^2}}
$$
is the well-known mean curvature operator and can be viewed as the flux which satisfies 
\[
\psi(s) \rightarrow \pm 1 \hbox{ as } 
s \rightarrow \pm \infty. 
\]
The bistable Rosenau equation~(\ref{eq:Ros}) with the physically relevant Neumann boundary conditions is studied in \cite{BG} where in particular it is shown that the bifurcation structure for the stationary problem associated with~(\ref{eq:Ros}) depends on the parameter $\epsilon$ as well as on the length $L$ of the interval $\Omega$. In fact, one can prove 

\begin{proposition}\cite[Proposition 3.1]{BG}
The $k$-th bifurcation from the trivial solution of the stationary problem associated with~(\ref{eq:Ros}) is a supercritical pitchfork if $L>k\pi/\sqrt{2}$ and a subcritical pitchfork if the inequality is reversed.
\end{proposition}

Through an analysis of the {\em time map} (see \cite{SW}) associated with the stationary problem for~(\ref{eq:Ros}), one can also show that, for any given value of $L$, there is a value of $\epsilon=\epsilon^*(L)$ such that for $\epsilon< \epsilon^*(L)$ there are classical (i.e. $C^2((0,L))\cap C^1([0,L])$) solutions to that problem and that solutions at $\epsilon = \epsilon^*(L)$ develop infinite gradient (see \cite{BG} for details).

Since the physical arguments of Rosenau \cite{ROSE} for assuming the free energy of a system to have the form (\ref{eq:energy}) are compelling, it makes sense to study a mass-conserving version of~(\ref{eq:Ros}). There are many ways of doing that, the simplest being to consider the mass-constrained $L^2$-gradient flow of (\ref{eq:energy}) which leads to 

\begin{equation}\label{eq:massconrose}
u_t\!=\!\epsilon\left(\psi(u_x)\right)_x+f(u)-\frac{1}{|\Omega|}\int_\Omega f(u)\,dx,\,\,\,x\in\Omega,\,\,t>0
\end{equation}
where $|\Omega|$ is the length of $\Omega$ and $\epsilon$, $\psi(s)$ and $f(u)$ are as above, with the Neumann boundary conditions $u_x=0$ on $\partial\Omega$ and a suitable initial condition $u(x,0)=u_0(x)$. Clearly, 
\[
\int_\Omega u(x,t)\,dx=\int_\Omega u_0(x)\,dx\hspace{0.2in}
\]
Thus (\ref{eq:massconrose}) is the quasilinear analogue of the Rubinstein-Sternberg equation from \cite{RS}. Local mass-conserving versions of (\ref{eq:Ros}) can also be introduced leading to a Cahn-Hilliard type equation, 
\begin{eqnarray}\label{eq:CHRose}
 u_t\!&=&\!-(\epsilon(\psi(u_x))_x+f(u))_{xx},\,\,\,x\in\Omega,\,\,t>0,\\
u_x\!&=&\!(\epsilon(\psi(u_x))_x+f(u))_x=0,\,\,\,x\in\partial\Omega,\nonumber
\end{eqnarray}
which can be obtained as the $H^{-1}$-gradient flow of the free energy functional in~(\ref{eq:energy}). Note that one can integrate the stationary problem associated with~(\ref{eq:CHRose}) twice with respect to $x$ using the Neumann boundary conditions to obtain the stationary problem associated with~(\ref{eq:massconrose}). Hence stationary solutions to~(\ref{eq:massconrose}) are also stationary solutions to~(\ref{eq:CHRose}).

\section{The stationary problem}\label{sect:statprob}

We are interested in this section in characterising the multiplicity of solutions to the stationary problem associated with~(\ref{eq:massconrose}) and hence also with (\ref{eq:CHRose}),
\begin{eqnarray}\label{eq:masscon}
\left(\psi(u_x)\right)_x &+&\lambda f(u)-\frac{\lambda}{L}\int_0^Lf(u(x))\,dx=0,\,\,\,\,\,\,\,x\in(0,L), \\
 u_x\!&=&\!0,\hbox{ at } x=0,\,,L, \; \;  \frac{1}{L}\int_0^Lu(x)\,dx\!=\!M,\nonumber
\end{eqnarray}
as the bifurcation parameter $\lambda=\frac{1}{\epsilon}$ varies in $(0,\infty)$. 

\subsection*{Liapunov-Schmidt reduction}

We use the Liapunov-Schmidt reduction \cite{GS} to obtain local bifurcation results for (classical) solutions to the non-local equation in~(\ref{eq:masscon}). This work is a quasilinear analogue to work done in \cite{EFG} which considers the case where $\psi(s)=s$ in~(\ref{eq:masscon}) and uses local bifurcation and path-following methods to examine the changes in bifurcation diagrams of stationary solutions to the Cahn-Hilliard model of phase separation as the mass constraint is varied.

Note that if $u(x)$ is a solution to~(\ref{eq:masscon}) then $u(L-x)$ is also a solution to~(\ref{eq:masscon}) and so bifurcations from the trivial solution $u(x)=M$ of~(\ref{eq:masscon}) arise as pitchforks. \\

We set $v=u-M$ and recast problem~(\ref{eq:masscon}) as $G(v,\lambda,M)=0$, where
\begin{eqnarray}\label{eq:G}
 G(v,\lambda,M)\!=\!\left(\frac{v_x}{\sqrt{1+(v_x)^2}}\right)_x\!+\!\lambda f(v+M)\!-\!\frac{\lambda}{L}\int_0^Lf(v(x)+M)\,dx.
\end{eqnarray}
Hence we regard $G$ as an operator $G:D(G)\subset H\rightarrow H$ where $D(G)$ is given by
\begin{eqnarray}
 D(G)=\left\{v\in C^2((0,L))\,\,:\,\,v'(0)=v'(L)=0,\,\,\frac{1}{L}\int_\Omega v(x)\,dx=0\right\},\nonumber
\end{eqnarray}
and $H$ is the space
\begin{eqnarray}
H=\left\{w\in C((0,L))\,:\,\frac{1}{L}\int_0^Lw(x)\,dx=0\right\}.\nonumber
\end{eqnarray}
The linearisation about the trivial solution $v=0$ is given by
\begin{eqnarray}
 (dG)_{0,\lambda,M}\cdot w\!&=&\!\frac{d}{dh}G(0+hw,\lambda,M)|_{h=0}\nonumber \\
\!&=&\! w_{xx}+\lambda f'(M)w-\frac{\lambda}{L}\int_0^Lf'(M)w(x)\,dx\nonumber \\
 \!&=&\!w_{xx}+\lambda f'(M)w,\nonumber 
\end{eqnarray}
since $w\in D(G)$. Hence $\textup{ker}(dG)_{0,\lambda,M}$ is one-dimensional when $\lambda=\lambda_k=\frac{k^2\pi^2}{L^2f'(M)}$ and is spanned by $v_k=\cos\left(\frac{k\pi x}{L}\right)$. Thus in a neighbourhood of a bifurcation point $(\lambda_k,0)$ of~(\ref{eq:masscon}), we show that solutions of $G(v,\lambda,M)=0$ on $H$ are in one-to-one correspondence with solutions of the reduced equation $h(\lambda,y)=0$, $y \in \mathbb{R}$, obtained through a Liapunov-Schmidt reduction. 

For the details of the computations, the reader is referred to the Appendix. For the derivatives of the bifurcation function $h(\lambda,y)$ evaluated at $\lambda=\lambda_k$, $y=0$ we obtain from (\ref{eq:derivatives of h})
\begin{eqnarray}
h=h_y=h_{yy}=h_\lambda=0,\nonumber 
\end{eqnarray}
Also, from (\ref{eq:hlambday}), $h_{\lambda y}= Lf'(M)$ and $h_{yyy}$ is given by (\ref{eq:hyyy}). 

Let us consider first the case $M=0$. Then $h_{\lambda y}=L>0$ by~(\ref{eq:hlambday}) and from~(\ref{eq:hyyy}) 
\[
h_{yyy}=\frac{3k^2\pi^2}{4L^3}(3k^2\pi^2-6L^2),
\]
so that the bifurcation from the trivial solution will be subcritical if $L<\frac{k\pi}{\sqrt{2}}$ and supercritical if $L>\frac{k\pi^2}{\sqrt{2}}$. This corresponds with the result mentioned in Section~\ref{sect:intro} and obtained in \cite[Proposition 3.1]{BG} for the non-conserving Neumann stationary problem for~(\ref{eq:Ros}).

Now suppose $0<|M|<\frac{1}{\sqrt{5}}$. Using~(\ref{eq:hyyy}), we see that $h_{yyy}>0$ as long as
\begin{equation}\label{Lstar}
 L< L^*:=\frac{k\pi}{\sqrt{2}}\frac{[1-3M^2]}{\sqrt{1-5M^2}},
\end{equation}
which has a vertical asymptote when $|M|=\frac{1}{\sqrt{5}}$ and a turning point when $|M|=\frac{1}{\sqrt{15}}$. We plot the relationship between $L^*$ in~(\ref{Lstar}) and $M$ for some $k\in\mathbb{N}$ in Figure~\ref{fig:lvsm}.\\ 
\begin{figure}[H]
\begin{center}
\scalebox{0.44}{\psfrag{L}{\Huge $L^*$}
\psfrag{0}{\Huge $0$}
\psfrag{1}{\Huge $\frac{k\pi}{\sqrt{2}}$}
\psfrag{2}{\Huge $\frac{2\sqrt{3}k\pi}{5}$}
\psfrag{3}{\Huge $\frac{1}{\sqrt{15}}$}
\psfrag{4}{\Huge $\frac{1}{\sqrt{5}}$}
\psfrag{5}{\Huge $\frac{1}{\sqrt{3}}$}
\psfrag{6}{\Huge $-\frac{1}{\sqrt{15}}$}
\psfrag{7}{\Huge $-\frac{1}{\sqrt{5}}$}
\psfrag{8}{\Huge $-\frac{1}{\sqrt{3}}$}
\psfrag{g1}{\Huge Super}
\psfrag{g2}{\Huge Sub}
\psfrag{M}{\Huge $M$}
\includegraphics{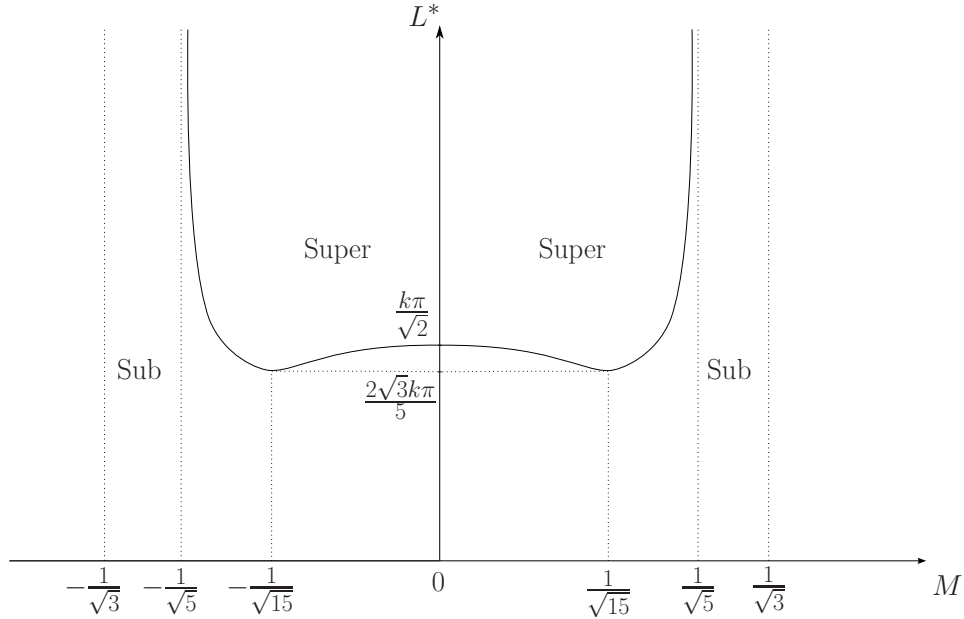}}
\end{center}
\caption{Plot of length $L^*$ against mass $M$ for some $k\in\mathbb{N}$, ($k=1$).} 
\label{fig:lvsm}
\end{figure}

Also, we note that for all $L>0$ with $\frac{1}{\sqrt{5}}<|M|<\frac{1}{\sqrt{3}}$,
\begin{eqnarray}
&&3k^2\pi^2[f'(M)]^2+L^2f'''(M)f'(M)+\frac{L^2}{3}[f''(M)]^2\nonumber \\
 \!&=&\! 3k^2\pi^2(1-3M^2)^2+6L^2(5M^2-1)>0,\nonumber 
\end{eqnarray}
so that from~(\ref{eq:hyyy}), $h_{yyy}>0$ for all for all such $L$ and $M$. Therefore, since $h_{\lambda y}=f'(M)L>0$ for all $M$ such that $|M|<\frac{1}{\sqrt{3}}$, we have established:
\begin{proposition}\label{prop:mc direction of pitchfork}
For $0\le |M|<\frac{1}{\sqrt{5}}$, the $k$-th bifurcation from the trivial solution of~(\ref{eq:masscon}) is a supercritical pitchfork if $L>L^*$ and a subcritical pitchfork if $L<L^*$ (with $L^*$ given in~(\ref{Lstar})). For $\frac{1}{\sqrt{5}}<|M|<\frac{1}{\sqrt{3}}$, bifurcation from $u(x)=M$ is a subcritical pitchfork for all $L$. In the parameter regime $|M|\ge\frac{1}{\sqrt{3}}$, there are no bifurcations from the constant solution $u(x)=M$ for any $L$.
\end{proposition}

\textbf{Remark}: Note that, unlike in the semilinear case of \cite{EFG} wherein $\psi(s)=s$, one can have both supercritical and subcritical pitchforks for different values of $k$.

\section{Nonexistence of classical solutions for $\lambda$ large enough}

In this section we show that as $\lambda \rightarrow \infty$, branches of {\em classical} stationary solutions do \textbf{not} persist indefinitely. To do that, we introduce an auxiliary problem. For fixed $L$ and $a\in\mathbb{R}$, consider
 \begin{eqnarray}\label{eq:ancprob}
  \left(\frac{u'}{\sqrt{1+(u')^2}}\right)'+\lambda f(u)\!&=&\!\lambda a,\\
    u'(0)=u'(L)\!&=&\!0,\nonumber
   \end{eqnarray}
which can be written as the first order system 
\begin{equation}\label{system}
\left\{ 
\begin{array}{ll}
u'&=v \\
v'&=\lambda(a-f(u))(1+v^2)^{\frac{3}{2}}
\end{array} \right.,
\end{equation}
with the first integral given by
\begin{eqnarray}\label{eq:firstintegral}
 H(u,v)=\frac{1}{\sqrt{1+v^2}}+\lambda(au-F(u)).
\end{eqnarray}
The connection between (\ref{eq:masscon}) and (\ref{eq:ancprob}) is seen by integrating (\ref{eq:ancprob}) over $\Omega=(0,L)$ to obtain
\[
\frac{1}{L}\int_0^L f(u(x))\,dx=a.
\]

\begin{lemma}\label{lem:nomoreclassical}
For a fixed $L$ and any $a\in\mathbb{R}$ such that $|a|<\frac{2}{3\sqrt{3}}$, there is a number $\lambda^*(a,L)$ such that $\forall \lambda>\lambda^*(a,L)$, there are no non-constant classical solutions to~(\ref{eq:ancprob}). 
\end{lemma}

\medskip

\begin{proof}
Suppose $0<a<\frac{2}{3\sqrt{3}}$ so that the equation $f(u)=a$ has three solutions $u_l$, $c$, $u_r$ with $u_l<0<c<u_r$ where $(u_l,0)$ and $(u_r,0)$ are saddle points for~(\ref{system}) and $(c,0)$ is a centre. From considerations on the first integral~(\ref{eq:firstintegral}) associated with the ancillary problem~(\ref{eq:ancprob}) one can see that for $\lambda<\lambda_h(a)$, there is a homoclinic loop connecting $(u_r,0)$ to itself which, in the case of a particular $\lambda<\lambda_h(a)$, we have denoted by $\gamma_\lambda$ in Figure~\ref{fig:homoclinicloop}. Note that for each $a$, $\lambda_h(a)$ is obtained through solving 
\begin{eqnarray}
 H(u_r,0)=H(c,-\infty),\nonumber
\end{eqnarray}
for $\lambda=\lambda_h(a)$. Non-constant classical solutions to~(\ref{eq:ancprob}) for $\lambda<\lambda_h(a)$ are represented in the phase plane by trajectories which encircle the centre $(c,0)$, start and end on the $u$-axis and which are {\em contained within} the homoclinic loop $\gamma_\lambda$.
\begin{figure}[H]
\begin{center}
\scalebox{0.43}{\psfrag{ul}{\Huge $u_l$}
\psfrag{ur}{\Huge $u_r$}
\psfrag{c}{\Huge $c$}
\psfrag{u1}{\Huge $u_{\lambda*}$}
\psfrag{g}{\Huge $\gamma_\lambda$}
\includegraphics{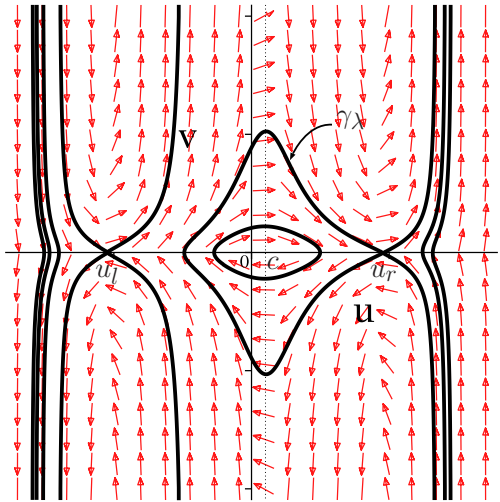}}
\end{center}
\caption{Phase portrait for~(\ref{system}) with $\lambda<\lambda_h(a)$ ($a=0.1$, $\lambda=3<\lambda_h(0.1)=6.3426$).}
\label{fig:homoclinicloop}
\end{figure}
For $\lambda$ large enough, i.e. for $\lambda=\lambda_h(a)$, the homoclinic loop will break. Through further considerations on $H(u,v)$, one sees that if $\lambda>\lambda_h(a)$, as in Figure~\ref{fig:brokehomoclinicloop}, there exist values $u^*(\lambda)$ and $u^{**}(\lambda)$ such that $u^*(\lambda)<c<u^{**}(\lambda)$, 
\begin{eqnarray}\label{eq:ustars}
 H(u^*(\lambda),0)=H(u^{**}(\lambda),0),
\end{eqnarray}
and
\begin{eqnarray}\label{eq:ustarstoc}
 u^*(\lambda)\rightarrow c^-,\,\,u^{**}(\lambda)\rightarrow c^+\hspace{0.1cm}\textup{as}\,\,\lambda\rightarrow\infty.
\end{eqnarray}
Note that for each $\lambda>\lambda_h(a)$, we can find $u^*(\lambda)$ by solving
\begin{eqnarray}
 H(u^*(\lambda),0)=H(c,\infty),\nonumber
\end{eqnarray}
for $u_l<u^*(\lambda)<c$ and then obtain $u^{**}(\lambda)$ via~(\ref{eq:ustars}). Hence non-constant classical solutions to~(\ref{eq:ancprob}) are now confined to $\gamma^*_\lambda$ as in Figure~\ref{fig:brokehomoclinicloop} where $\gamma^*_\lambda$ is the region in the phase plane enclosed by the trajectories through $(u^*(\lambda),0)$ and through $(u^{**}(\lambda),0)$ for $\lambda>\lambda_h(a)$.
\begin{figure}[H]
\begin{center}
\scalebox{0.45}{\psfrag{ul}{\Huge $u_l$}
\psfrag{ur}{\Huge $u_r$}
\psfrag{c}{\Huge $c$}
\psfrag{ul}{\Huge $u_l$}
\psfrag{u1}{\Large $u^*(\lambda)$}
\psfrag{u2}{\Large $u^{**}(\lambda)$}
\psfrag{g}{\Huge $\gamma^*_\lambda$}
\includegraphics{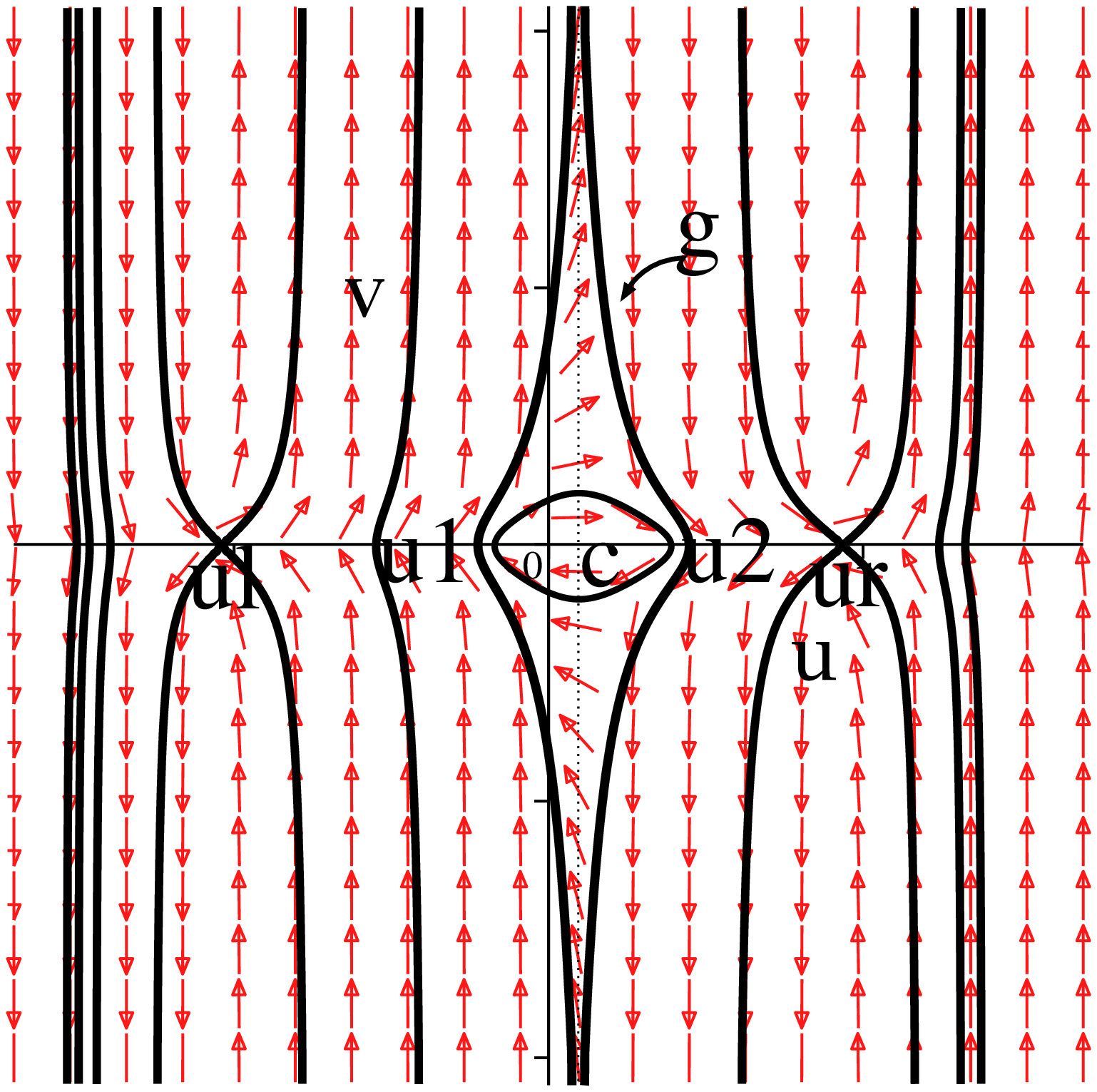}}
\end{center}
\caption{Phase portrait for~(\ref{system}) with $\lambda>\lambda_h(a)$, ($a=0.1$, $\lambda=15>\lambda_h(0.1)\simeq6.3426$).} 
\label{fig:brokehomoclinicloop}
\end{figure}
Now assume that for a given $\lambda$ and a given $n\in\mathbb{N}$, there exists a non-constant classical solution $u(x)$ to~(\ref{eq:ancprob}) which has $n$ points of inflection in $(0,L)$. Such a solution will have period equal to $\frac{2L}{n}$. For $\lambda$ large enough, because of~(\ref{eq:ustarstoc}), the only possible non-constant Neumann classical solutions to~(\ref{eq:ancprob}) will be in a small neighbourhood of the centre, hence we linearise~(\ref{eq:ancprob}) around $(c,0)$ to obtain
\begin{eqnarray}
 u''+\lambda f'(c)u=0.\nonumber
\end{eqnarray}
Thus the period of such a solution is equal to 
\begin{eqnarray}\label{eq:period}
 \frac{2\pi}{\sqrt{\lambda f'(c)}},
\end{eqnarray}
which is well-defined since $0<c<\frac{1}{\sqrt{3}}$ so that $f'(c)>0$ but~(\ref{eq:period}) will not be equal to $\frac{2L}{n}$ if $\lambda$ is large enough. This contradiction proves Lemma~\ref{lem:nomoreclassical} in the case that $ 0<a<\frac{2}{3\sqrt{3}}$. The case $-\frac{2}{3\sqrt{3}}<a<0$ can be treated similarly. Note that in the case that $a=0$ (for which non-constant classical solutions to~(\ref{eq:ancprob}) are also stationary solutions of~(\ref{eq:Ros})), there are heteroclinic loops connecting saddle points which break for $\lambda$ large enough rather than homoclinic loops but the arguments carry through just as easily.
\end{proof}

\textbf{Remark}: We note that for each $0<a<\frac{2}{3\sqrt{3}}$, if $\lambda$ is small enough then the phase portrait for~(\ref{system}) contains non-trivial classical solutions to the ancillary problem~(\ref{eq:ancprob}) which are also non-trivial classical solutions to~(\ref{eq:masscon}) for some $0<M<1$.  Similarly, for each $-\frac{2}{3\sqrt{3}}<a<0$, if $\lambda$ is small enough then the phase portrait for~(\ref{system}) contains non-trivial classical solutions to~(\ref{eq:ancprob}) which are also non-trivial classical solutions to~(\ref{eq:masscon}) for some $-1<M<0$. If $a=0$ then for $\lambda$ small enough, the phase portrait for~(\ref{system}) contains non-trivial classical solutions to~(\ref{eq:ancprob}) which are also non-trivial classical solutions to~(\ref{eq:masscon}) with $M=0$. \\ 

 We now concentrate on the multiplicity of {\em monotone} classical solutions to~(\ref{eq:masscon}) where without loss of generality, $0\le M<\frac{1}{\sqrt{3}}$ so that from now on, for $a\in\left[0,\frac{2}{3\sqrt{3}}\right)$, $\lambda^*(a,L)$ represents the value of $\lambda$ for which there are no monotone classical solutions to~(\ref{eq:ancprob}) $\forall \lambda>\lambda^*(a,L)$. Note that $\lambda^*(a,L)\rightarrow\lambda^*(0,L)$ as $a\rightarrow0$ and since the value $\lambda^*(a,L)$ must occur after the homoclinic orbit has broken for a fixed $a$, we also have that $\lambda^*(a,L)\rightarrow\infty$ as $a\rightarrow\frac{2}{3\sqrt{3}}$.\\ 

 Let us drop the dependence of $\lambda^*(a,L)$ from Lemma~\ref{lem:nomoreclassical} on $L$ and denote the value simply by $\lambda^*(a)$ since we will regard $L$ as being fixed. Hence Lemma~\ref{lem:nomoreclassical} implies that for every $a\in\left[0,\frac{2}{3\sqrt{3}}\right)$, there is a value $\lambda^*(a)$ such that monotone solutions to the ancillary problem~(\ref{eq:ancprob}) develop infinite gradient at some $x_0\in(0,L)$. Therefore for every $\lambda\ge\lambda^*(0)$, there is a value of $a$, corresponding to the inverse of $\lambda^*(a)$, which will be denoted by $a^*(\lambda)$ and be such that $a^*(\lambda)\rightarrow\frac{2}{3\sqrt{3}}$ as $\lambda\rightarrow\infty$ and $a^*(\lambda)\rightarrow 0$ as $\lambda\rightarrow\lambda^*(0)^+$. We give a sketch of the function $a^*(\lambda)$ in Figure~\ref{fig:lambdastara} but note that we do not prove that the curve $a^*(\lambda)$ is continuous (see Theorem~\ref{thm:nomoreclassical}).
\begin{figure}[H]
 \begin{center}
  \scalebox{0.55}
{\psfrag{l}{\LARGE $\lambda$}
\psfrag{0}{\Huge $0$}
\psfrag{l2}{\LARGE $\lambda^*(0)$}
\psfrag{a}{\LARGE $a$}
\psfrag{a2}{\LARGE $\frac{2}{3\sqrt{3}}$}
\psfrag{a1}{\LARGE $a^*(\lambda)$}
\includegraphics{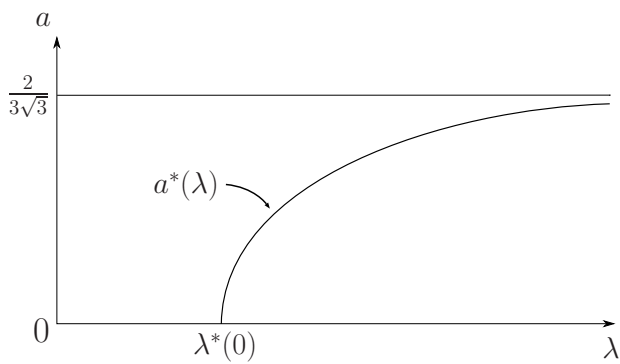}}
 \end{center}
\caption{Sketch of the curve $a^*(\lambda)$ in the $(\lambda,a)$-plane.}
\label{fig:lambdastara}
\end{figure}
\begin{lemma}\label{lem:a bar M}
For a fixed $L$ and each $0\le M<\frac{1}{\sqrt{3}}$, the bifurcation curve of monotone classical solutions to~(\ref{eq:masscon}) does not exist in the region in the $(\lambda,a)$-plane defined by $a^*(\lambda)<a<\frac{2}{3\sqrt{3}}$ for $\lambda$ large, where $a^*(\lambda)$ is the curve obtained from Lemma~\ref{lem:nomoreclassical}.
\end{lemma}
\begin{proof}
Let $\tilde{a}(\lambda)$ denote the bifurcation curve of monotone classical solutions to~(\ref{eq:masscon}) in the $(\lambda,a)$-plane and assume the contrary to Lemma~\ref{lem:a bar M}. Then there exists a sequence $(\lambda_n)_{n=1}^\infty$ such that $\lambda_n\rightarrow\infty$ as $n\rightarrow\infty$ and
\begin{eqnarray}
 a^*(\lambda_n)<\tilde{a}(\lambda_n)<\frac{2}{3\sqrt{3}},\nonumber
\end{eqnarray}
for all $n\in\mathbb{N}$ sufficiently large. Then, since $a^*(\lambda_n)\rightarrow\frac{2}{3\sqrt{3}}$ as $n\rightarrow\infty$, we have by the sandwich theorem that 
\begin{eqnarray}\label{eq:atilde}
 \tilde {a}(\lambda_n)\rightarrow\frac{2}{3\sqrt{3}}\hspace{0.1in}\textup{as}\,\,\,n\rightarrow\infty.
\end{eqnarray}
But $M$ is constant along $\tilde{a}(\lambda)$ and~(\ref{eq:atilde}) implies that $M$ must be equal to $\frac{1}{\sqrt{3}}$ which is a contradiction. 
\end{proof}

 We now have the following theorem regarding (monotone) classical solutions to the non-local mass-conserving equation~(\ref{eq:masscon}):
\begin{thm}\label{thm:nomoreclassical}
For fixed $L$ and $0\le M<\frac{1}{\sqrt{3}}$, there exists a value $\lambda_1(M,L)$ such that for $\lambda>\lambda_1(M,L)$ there cannot exist monotone classical solutions to the non-local mass-conserving equation in~(\ref{eq:masscon}).
\end{thm}
\begin{proof}
  In the $(\lambda,a)$-plane, the curve $a^*(\lambda)$ obtained from Lemma~\ref{lem:nomoreclassical} separates two regions; a region in which monotone classical solutions to the non-local mass-conserving equation~(\ref{eq:masscon}) can exist for fixed $L>0$ and $0\le M<\frac{1}{\sqrt{3}}$ and a region in which monotone classical solutions to~(\ref{eq:masscon}) cannot exist for such $L$ and $M$. As in Lemma~\ref{lem:a bar M}, let $\tilde{a}(\lambda)$ denote the bifurcation curve of monotone classical solutions to~(\ref{eq:masscon}). We know by Lemma~\ref{lem:a bar M}, that $\tilde{a}(\lambda)$ {\em must} intersect the curve $a^*(\lambda)$ at some point since $\tilde{a}(\lambda)$ cannot not exist in the region defined by $a^*(\lambda)<a<\frac{2}{3\sqrt{3}}$ in the $(\lambda,a)$-plane for sufficiently large $\lambda$ and from the Rabinowitz theorem (see \cite[Theorem 13.10]{SMOLLER}), it has to go {\em somewhere} as $\lambda$ increases. We want to prove that $\tilde{a}(\lambda)$ intersects $a^*(\lambda)$ at a point of continuity of $a^*(\lambda)$.  
\begin{figure}[H]
 \begin{center}
  \scalebox{0.36}
{\psfrag{l}{\Huge $\lambda$}
\psfrag{0}{\Huge $0$}
\psfrag{a}{\Huge $a$}
\psfrag{a1}{\Huge $a^*(\lambda)$}
\psfrag{a2}{\Huge $\tilde{a}(\lambda)$}
\psfrag{a3}{\Huge $\frac{2}{3\sqrt{3}}$}
\psfrag{C}{\Large $\textup{CLASSICAL SOLUTIONS}$}
\psfrag{C2}{\Large $\textup{TO~(\ref{eq:masscon})}$}
\psfrag{l2}{\Huge $\lambda^*(0,L)$}
\psfrag{l3}{\Huge $\lambda_d^*$}
\psfrag{N}{\Large $\textup{NO CLASSICAL SOLUTIONS}$}
\psfrag{B}{\Huge $\textup{Supposed bifurcation curve for~(\ref{eq:masscon})}$}
\psfrag{L}{\Huge $(\textup{for fixed \textit{L} and \textit{M}})$}
\includegraphics{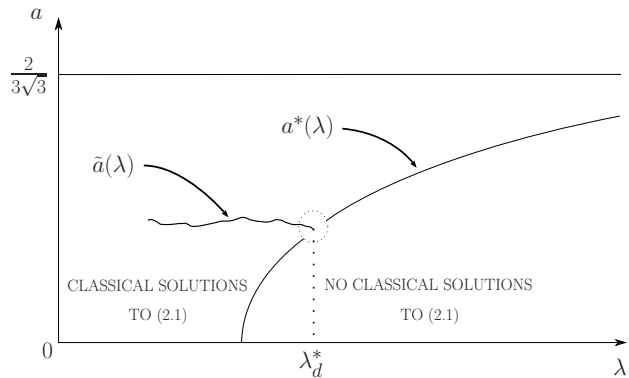}}
 \end{center}
\caption{The curve $a^*(\lambda)$ and a proposed bifurcation curve $\tilde{a}(\lambda)$ of classical monotone solutions to~(\ref{eq:masscon}) intersecting $a^*(\lambda)$ at a supposed point of discontinuity of $a^*(\lambda)$.} 
\label{lambdastara2}
\end{figure}
We assume the contrary and consider the bifurcation curve $\tilde{a}(\lambda)$ of monotone classical solutions to~(\ref{eq:masscon}) in the $(\lambda,a)$-plane as we have it in Figure~\ref{lambdastara2}. We assume that the curve $a^*(\lambda)$ is discontinuous at some value $\lambda_d^*$ and that the bifurcation curve $\tilde{a}(\lambda)$ enters the region of no classical solutions to~(\ref{eq:masscon}) at this point of discontinuity of $a^*(\lambda)$. By the Rabinowitz theorem \cite[Theorem 13.10]{SMOLLER} 
there must be a classical solution in a neighbourhood of this point $\lambda_d^*$ and the bifurcation curve of classical solutions can be extended. However, we would then have entered into the region in which there can be no classical solutions to~(\ref{eq:masscon}) and so we have obtained a contradiction. Hence the bifurcation curve must intersect the curve $a^*(\lambda)$ at a point of continuity and the theorem is proven.\\
\end{proof}
\textbf{Remark}: Although we established Theorem~\ref{thm:nomoreclassical} for monotone classical solutions to~(\ref{eq:masscon}), the result can be generalised to non-monotone classical solutions $u(x)$ to~(\ref{eq:masscon}) which have $n$ points of inflection in $(0,L)$ for some $n\in\mathbb{N}$. Thus for {\em each} $n\in\mathbb{N}$, there exists a value $\lambda_n(M,L)$ such that for $\lambda>\lambda_n(M,L)$, there are no classical solutions to~(\ref{eq:masscon}) with $n$ points of inflection in $(0,L)$. However, we point out that unlike the situation in the (semilinear) Rubinstein-Sternberg equation \cite{EFG}, there is no obvious way to deduce the behaviour of a particular branch of solutions to~(\ref{eq:masscon}) from that of the monotone one. Also, in the proof of Theorem~\ref{thm:nomoreclassical} we did not establish that the curve $a^*(\lambda)$ is continuous but we did show that it must be continuous at the point $\lambda_d^*$ where the bifurcation curve of monotone classical solutions to~(\ref{eq:masscon}) intersects it. \\

We have not given a precise depiction of the bifurcation curve of monotone classical solutions to the mass-conserving problem~(\ref{eq:masscon}) in the $(\lambda,a)$-plane for given $L$ and $0\le M<\frac{1}{\sqrt{3}}$. Hence for fixed $L$ and $M$ subcritical and then for fixed $L$ and $M$ supercritical, we numerically obtain the curve $a^*(\lambda)$ from Lemma~\ref{lem:nomoreclassical} and the bifurcation curve $\tilde{a}(\lambda)$ of monotone classical solutions to~(\ref{eq:masscon}) and plot these in the $(\lambda,a)$-plane. With regards to problem~(\ref{eq:masscon}), we know that along the line of trivial solutions, $a=\frac{1}{L}\int_0^Lf(u(x))\,dx$ is constant and equal to $M-M^3$. Also, given $L$, for each $0\le M<\frac{1}{\sqrt{3}}$, bifurcation from the trivial solution of~(\ref{eq:masscon}) occurs when $\lambda=\frac{\pi^2}{L^2f'(M)}$ which we can rearrange to
\begin{eqnarray}
 M=\frac{1}{\sqrt{3}}\sqrt{1-\frac{\pi^2}{L^2\lambda}}\nonumber,
\end{eqnarray}
and since bifurcation points appear along the line of trivial solutions (upon which $a=M-M^3$) we can plot a curve of bifurcation points for fixed $L$ in the $(\lambda,a)$-plane by considering the function
\begin{eqnarray}
 a_b(\lambda)\!&=&\!M-M^3\nonumber \\
\!&=&\!\frac{1}{3\sqrt{3}}\sqrt{1-\frac{\pi^2}{L^2\lambda}}\left(2+\frac{\pi^2}{L^2\lambda}\right).\nonumber
\end{eqnarray}
For a given $L$, we can numerically plot the curve $a^*(\lambda)$ in the $(\lambda,a)$-plane by solving~(\ref{eq:masscon}) for a variety of $M\in\left[0,\frac{1}{\sqrt{3}}\right)$ using path-following methods of AUTO \cite{AUTO} and determine for each $M$ the value of $\lambda$ and the value of  $a=\frac{1}{L}\int_0^Lf(u)\,dx$ for which the stationary solutions to~(\ref{eq:masscon}) develop infinite gradient in $(0,L)$ and then plot $a$ against $\lambda$. We can then use AUTO again for fixed $L$ and a particular $M\in\left[0,\frac{1}{\sqrt{3}}\right)$ to plot the bifurcation curve $\tilde{a}(\lambda)$  of monotone classical solutions to~(\ref{eq:masscon}) in the $(\lambda,a)$-plane and obtain the value $\lambda_1(M,L)$ of Theorem~\ref{thm:nomoreclassical}.\\

Thus we fix $L=1$ and take $M=0.1$ so that the bifurcation at $\lambda=\frac{\pi^2}{f'(0.1)}= 10.1749$ from the trivial solution is subcritical and $a=0.099$ along the line of trivial solutions to~(\ref{eq:masscon}). According to AUTO, the branch of monotone classical solutions to~(\ref{eq:masscon}) for these parameter values stops when $\lambda=\lambda_1(0.1,1)=5.6579$ with $\tilde{a}(\lambda_1(0.1,1))=a^*(\lambda_1(0.1,1))=0.0289$ and we have plotted all relevant curves for these values of $L$ and $M$ in the $(\lambda,a)$-plane in Figure~\ref{laplane}. 
\begin{figure}[H]
 \begin{center}
  \scalebox{0.7}
{\psfrag{l}{\large $\lambda$}
\psfrag{a}{\large $a$}
\psfrag{a1}{\large $0.099$}
\psfrag{a2}{\large $\tilde{a}(\lambda)$}
\psfrag{a3}{\large $a^*(\lambda)$}
\psfrag{a4}{\large $\frac{2}{3\sqrt{3}}$}
\psfrag{a5}{\large $a_b(\lambda)$}
\psfrag{0}{\large $0$}
\psfrag{5}{\large $5$}
\psfrag{10}{\large $10$}
\psfrag{15}{\large $15$}
\psfrag{20}{\large $20$}
\includegraphics{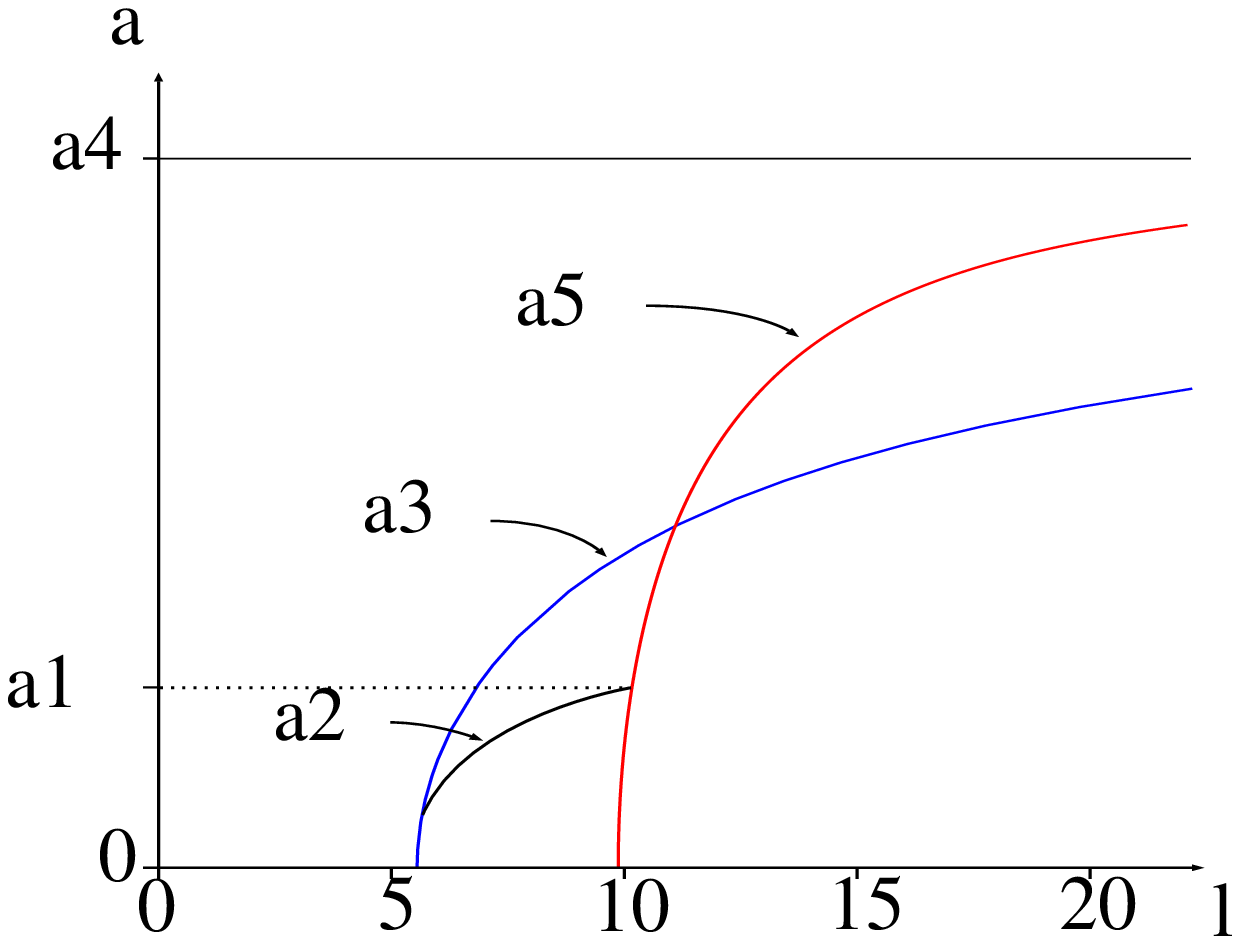}
}
 \end{center}
\caption{The bifurcation curve $\tilde{a}(\lambda)$ of monotone classical solutions to~(\ref{eq:masscon}) for $L=1$, $M=0.1$ and the curves $a_b(\lambda)$ and $a^*(\lambda)$ in the $(\lambda,a)$-plane.}
\label{laplane}
\end{figure}
Now consider $L=2.5$, $M=0.3$ for which $a=0.273$ along the line of trivial solutions to~(\ref{eq:masscon}). In Figure~\ref{laplane2} we plot all relevant curves in the $(\lambda,a)$-plane for these values of $L$ and $M$. We have a supercritical bifurcation from the trivial solution when $\lambda=\frac{\pi^2}{2.5^2f'(0.3)}= 2.1632$ and according to AUTO, the branch of monotone classical solutions in this case stops when $\lambda=\lambda_1(0.3,2.5)=4.0860$ with $\tilde{a}(\lambda_1(0.3,2.5))=a^*(\lambda_1(0.3,2.5))=0.0051$.
\begin{figure}[H]
 \begin{center}
  \scalebox{0.7}
{\psfrag{l}{\large $\lambda$}
\psfrag{a}{\large $a$}
\psfrag{a1}{\large $0.273$}
\psfrag{a2}{\large $\tilde{a}(\lambda)$}
\psfrag{a3}{\large $a^*(\lambda)$}
\psfrag{a4}{\large $\frac{2}{3\sqrt{3}}$}
\psfrag{a5}{\large $a_b(\lambda)$}
\psfrag{0}{\large $0$}
\psfrag{5}{\large $5$}
\psfrag{10}{\large $10$}
\psfrag{15}{\large $15$}
\psfrag{20}{\large $20$}
\includegraphics{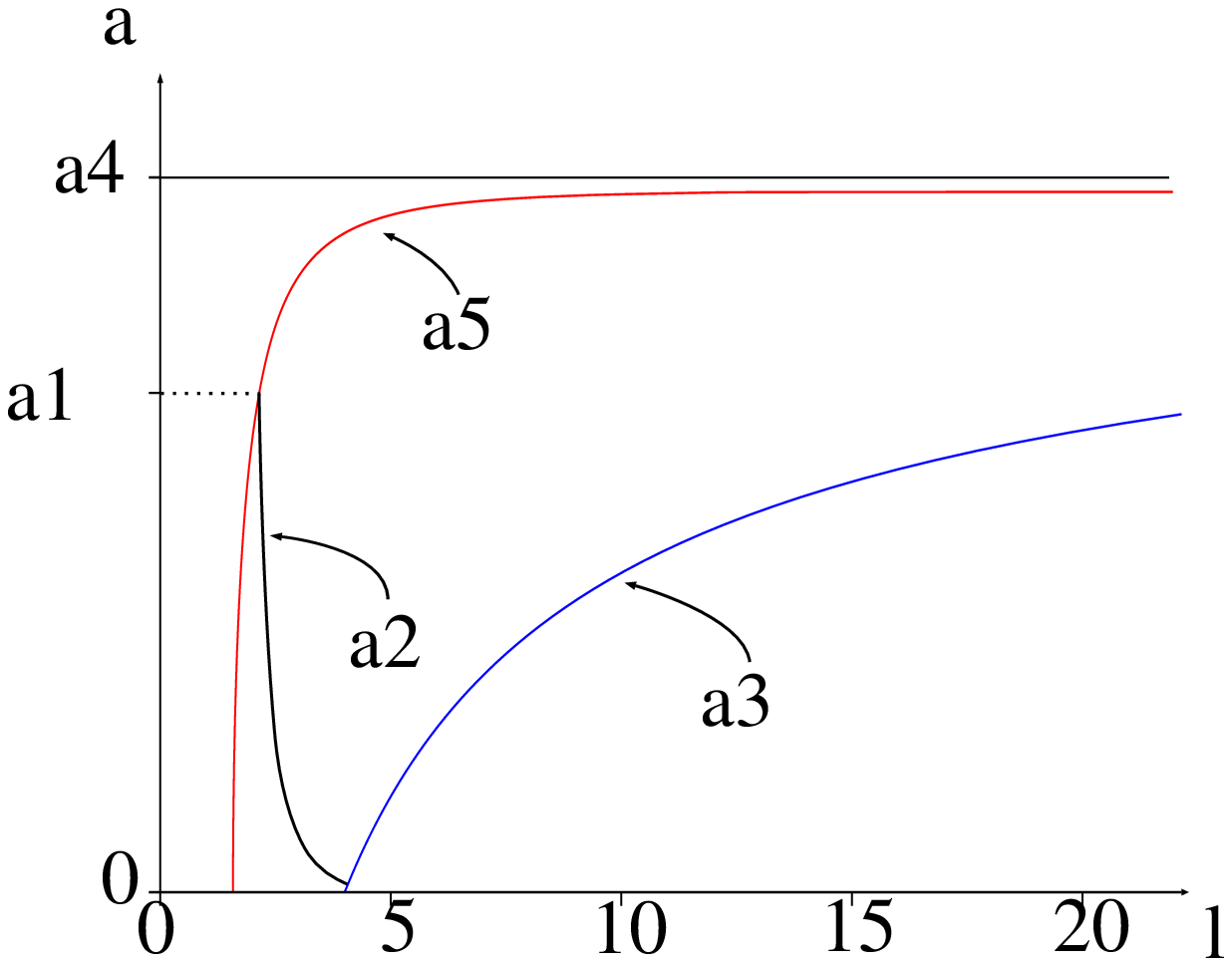}
}
 \end{center}
\caption{The bifurcation curve $\tilde{a}(\lambda)$ of monotone classical solutions to~(\ref{eq:masscon}) for $L=2.5$, $M=0.3$ and the curves $a_b(\lambda)$ and $a^*(\lambda)$ in the $(\lambda,a)$-plane.}
\label{laplane2}
\end{figure}
On the basis of the preceding numerical experiments we would conjecture that the bifurcation curve $\tilde{a}(\lambda)$ of (monotone) classical solutions to~(\ref{eq:masscon}) is monotonic and for given $L$ and $M$, $a$ is at its largest along the line of trivial solutions i.e. when $a=M-M^3$.\\

\textbf{Remark:} We have not discussed the case of having $\frac{1}{\sqrt{3}}<|M|<1$ in~(\ref{eq:masscon}), i.e. what happens in the ``metastable'' regime in which there are no local bifurcations from the trivial solution and for which the trivial solution is always stable. As $|M|\rightarrow\frac{1}{\sqrt{3}}$, $f'(M)\rightarrow0$ and the first and all subsequent bifurcation points $\lambda_k=\frac{k^2\pi^2}{L^2f'(M)}$ go off to infinity. In the semilinear situation studied in \cite{EFG}, in passing from the spinodal to the metastable regime, they show through spectral approximations and path-following methods that the saddle-nodes which exist for $\frac{1}{\sqrt{5}}\le |M|<\frac{1}{\sqrt{3}}$, move off to the right as $f'(M)\rightarrow0^+$ 
but at a speed much slower than that of the bifurcation points. In our case, for all $L$ with $M$ just less than $\frac{1}{\sqrt{3}}$ so that the bifurcation from the trivial solution is subcritical by Proposition~\ref{prop:mc direction of pitchfork}, the classical solutions to~(\ref{eq:masscon}) stop existing {\em before} we reach a saddle-node. Hence we were not able to perform a two parameter continuation in $\lambda$ and in $M$ of the saddle-nodes for $|M|$ beyond $\frac{1}{\sqrt{3}}$. We can however say what happens in the parameter regime $|M|\ge1$. By the remark after Lemma~\ref{lem:nomoreclassical} we see that it is not possible to construct a non-trivial classical solution to the ancillary problem~(\ref{eq:ancprob}) which will have average mass $M$ such that $|M|\ge1$. One can also see from phase portraits associated with~(\ref{system}) that it is also not possible to construct a non-classical solution to~(\ref{eq:ancprob}) for any $|a|<\frac{2}{3\sqrt{3}}$ which will have $|M|\ge1$. Therefore for $|M|\ge1$, there are no non-trivial solutions to~(\ref{eq:masscon}) for any $\lambda\in(0,\infty)$ which is also true of the semilinear problem (see \cite{EFG}).

\section{Numerics} In this section we carry out some numerical experiments for the non-local mass-conserving equation in~(\ref{eq:massconrose}) in order top see what (stable) non-classical solutions of~(\ref{eq:massconrose}) look like. We first derive an explicit numerical method to solve the equation~(\ref{eq:massconrose}) and we are grateful to Dr John Mackenzie of The University of Strathclyde for discussions on the numerical scheme we outline in the following subsection.
\subsection*{Numerical approximation}\label{sect:numerical scheme} 
We derive a mass-conserving numerical scheme to solve
\begin{eqnarray}\label{problem}
 u_t\!&=&\!\left(\frac{u_{x}}{\sqrt{1+u_x^2}}\right)_x+\lambda f(u)-\frac{\lambda}{|\Omega|}\int_\Omega f(u)\,dx,\hspace{0.1in}(x,t)\in \Omega\times(0,T) \nonumber \\
 u_x(0,t)\!&=&\!u_x(L,t)=0,\hspace{0.15in}(x,t)\in\partial \Omega\times(0,T)  \\
 u(x,0)\!&=&\!u_0(x),\hspace{0.5in}x\in\Omega,\nonumber
\end{eqnarray}
where $\lambda\in(0,\infty)$, $\Omega=(0,L)$, $L>0$, $f(s)=s-s^3$ and $0<T<\infty$. We obtain~(\ref{problem}) from~(\ref{eq:massconrose}) by multiplying~(\ref{eq:massconrose}) by $\frac{1}{\epsilon}$ and scaling time as $t\mapsto \epsilon t$. One can easily show that the non-local equation~(\ref{problem}) conserves mass and to hope to have the same at the discrete level we must discretise the equation in conservative form, i.e. in the form given in~(\ref{problem}). \\

We discretise the space interval $\Omega$ into $N+1$ evenly spaced points
\begin{equation}
 0=x_1<x_2<\ldots<x_{N+1}=L,\nonumber
\end{equation}
so that $\Delta x=\frac{L}{N}$ and we regard there as being cells $[x_{i-\frac{1}{2}},x_{i+\frac{1}{2}}]$ of width $\Delta x$ around each {\em internal} point $x_i$, $i=2,\cdots,N$ while at the boundaries $i=1$ and $i=N+1$ we use cells $[x_1,x_{\frac{3}{2}}]$ and $[x_{N+\frac{1}{2}},x_{N+1}]$ of half width $\frac{\Delta x}{2}$ as in Figure~\ref{discretize}.
\begin{figure}[H]
\begin{center}
\scalebox{0.5}{\psfrag{i1}{\Huge $x_1$}
\psfrag{i1.5}{\Huge $x_{2-\frac{1}{2}}$}
\psfrag{i2}{\Huge $x_2$}
\psfrag{i2.5}{\Huge $x_{2+\frac{1}{2}}$}
\psfrag{i3}{\Huge $x_3$}
\psfrag{i4}{\Huge $x_4$}
\psfrag{iN-0.5}{\Huge $x_{N-\frac{1}{2}}$}
\psfrag{iN}{\Huge $x_N$}
\psfrag{iN.5}{\Huge $x_{N+\frac{1}{2}}$}
\psfrag{iN+1}{\Huge $x_{N+1}$}
\psfrag{0}{\Huge $0$}
\psfrag{L}{\Huge $L$}
\psfrag{dots}{\Huge $\cdots$}
\includegraphics{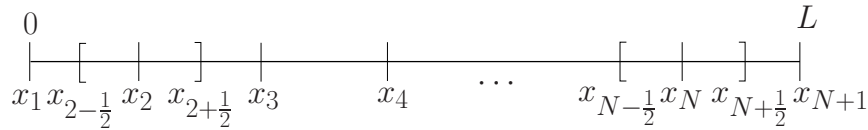}}
\end{center}
\caption{Discretisation of the space interval $[0,L]$.} 
\label{discretize}
\end{figure}
 At a particular $x_i\in\Omega=(0,L)$ (hence $i\in\{2\cdots N\}$), $t_n\in(0,T)$ the equation is given by
\begin{eqnarray}
 u_t(x_i,t_n)=\left(\frac{u_x(x_i,t_n)}{\sqrt{1+u_x^2(x_i,t_n)}}\right)_x+\lambda f(u(x_i,t_n))-\frac{\lambda}{L}\int_0^Lf(u(x,t_n))\,dx.\nonumber
\end{eqnarray}
Let $v(x,t)$ denote the flux, i.e.
\begin{eqnarray}
 v(x,t)=\frac{u_x(x,t)}{\sqrt{1+u_x^2(x,t)}}\nonumber
\end{eqnarray}
so that at some $x_i\in\Omega=(0,L)$, $t_n\in(0,T)$
\begin{eqnarray}\label{eq:discretized flux term}
 \left(\frac{u_x(x_i,t_n)}{\sqrt{1+u_x^2(x_i,t_n)}}\right)_x\!&=&\!v_x(x_i,t_n)\nonumber \\
 \!&\simeq&\!\frac{v(x_{i+\frac{1}{2}},t_n)-v(x_{i-\frac{1}{2}},t_n)} {\Delta x}
\end{eqnarray}
that is, we approximate the flux term by a first central difference in space at the point $x_i$ with half-spacing.\\

 Now,  
\begin{eqnarray}\label{eq:xi+1/2}
 v(x_{i+\frac{1}{2}},t_n)\!&=&\!\frac{u_x(x_{i+\frac{1}{2}},t_n)}{\sqrt{1+u_x^2(x_{i+\frac{1}{2}},t_n)}}\nonumber \\
 \!&\simeq&\! \frac{(u(x_{i+1},t_n)-u(x_i,t_n))/\Delta x}{\sqrt{1+\left[\frac{u(x_{i+1},t_n)-u(x_i,t_n)}{\Delta x}\right]^2}}
\end{eqnarray}
and 
\begin{eqnarray}\label{eq:xi-1/2}
  v(x_{i-\frac{1}{2}},t_n)\!&=&\!\frac{u_x(x_{i-\frac{1}{2}},t_n)}{\sqrt{1+u_x^2(x_{i-\frac{1}{2}},t_n)}}\nonumber \\
 \!&\simeq&\! \frac{(u(x_i,t_n)-u(x_{i-1},t_n))/\Delta x}{\sqrt{1+\left[\frac{u(x_i,t_n)-u(x_{i-1},t_n)}{\Delta x}\right]^2}}
\end{eqnarray}
so that in~(\ref{eq:xi+1/2}) we approximate the term $v(x_{i+\frac{1}{2}},t_n)$ by a first central difference in space about the point $x_{i+\frac{1}{2}}$ with half-spacing and in~(\ref{eq:xi-1/2}) we approximate the term $v(x_{i-\frac{1}{2}},t_n)$ by a first central difference in space about the point $x_{i-\frac{1}{2}}$ with half-spacing. Therefore by~(\ref{eq:discretized flux term}),~(\ref{eq:xi+1/2}) and~(\ref{eq:xi-1/2})
\begin{eqnarray}
 \left(\frac{u_x(x_i,t_n)}{\sqrt{1+u_x^2(x_i,t_n)}}\right)_x\!&=&\!v_x(x_i,t_n)\nonumber \\
 &\simeq&\frac{1}{\Delta x^2}\left(\frac{(u(x_{i+1},t_n)-u(x_i,t_n))}{\sqrt{1+\left[\frac{u(x_{i+1},t_n)-u(x_i,t_n)}{\Delta x}\right]^2}}-\frac{(u(x_i,t_n)-u(x_{i-1},t_n))}{\sqrt{1+\left[\frac{u(x_i,t_n)-u(x_{i-1},t_n)}{\Delta x}\right]^2}}\right).
 \nonumber
\end{eqnarray}
We approximate the integral term in~(\ref{problem}) at some given time $t_n$ using the midpoint rule for numerical integration as follows
\begin{eqnarray}
 \int_0^L f(u(x,t_n))\,dx\!&=&\!\int_{x_1}^{x_{\frac{3}{2}}}f(u(x,t_n))\,dx+\int_{x_{\frac{3}{2}}}^{x_{\frac{5}{2}}}f(u(x,t_n))\,dx+\ldots+\int_{x_{i-\frac{1}{2}}}^{x_{i+\frac{1}{2}}}f(u(x,t_n))\,dx\nonumber \\
&&\hspace{0.1in}+\ldots+\int_{x_{N-\frac{1}{2}}}^{x_{N+\frac{1}{2}}}f(u(x,t_n))\,dx+\int_{x_{N+\frac{1}{2}}}^{x_{N+1}}f(u(x,t_n))\,dx\nonumber \\
 &\simeq&\frac{\Delta x}{2}f(u(x_1,t_n))+\sum_{j=2}^N\Delta xf(u(x_j,t_n))+\frac{\Delta x}{2}f(u(x_{N+1},t_n)).\nonumber
\end{eqnarray}
Let $u(x_i,t_n)=u_i^n$ so that using the forward Euler method, we approximate~(\ref{problem}) at interior points $x_i\in(0,L)$, $i=2,\cdots,N$ by the numerical scheme
\begin{eqnarray}\label{eq:interior}
 u_i^{n+1}=u_i^n&+\Delta t\Bigg(\frac{1}{\Delta x^2}\left[\frac{u_{i+1}^n-u_i^n}{\sqrt{1+\left(\frac{u_{i+1}^n-u_i^n}{\Delta x}\right)^2}}-\frac{u_{i}^n-u_{i-1}^n}{\sqrt{1+\left(\frac{u_{i}^n-u_{i-1}^n}{\Delta x}\right)^2}}\right]+\lambda f(u_i^n)\nonumber \\
 &-\frac{\lambda}{L}\left[\frac{\Delta x}{2}f(u_1^n)+\sum_{j=2}^N\Delta xf(u_j^n)+\frac{\Delta x}{2}f(u_{N+1}^n)\right]\Bigg),
\end{eqnarray}
for $i=2,\ldots,N$. At the boundary points $x=0$ and $x=L$ we consider the discretisation in the half-cells $[x_1,x_{\frac{3}{2}}]$ and $[x_{N+\frac{1}{2}},x_{N+1}]$ respectively, so that we have
\begin{eqnarray}\label{eq:leftboundary}
 u_1^{n+1}=u_1^{n}\!&+&\!\Delta t\Bigg(\frac{2}{\Delta x^2}\frac{u_2^n-u_1^n}{\sqrt{1+\left[\frac{u_2^n-u_1^n}{\Delta x}\right]^2}}+\lambda f(u_1^n)\nonumber \\
\!&-&\!\frac{\lambda}{L}\left[\frac{\Delta x}{2}f(u_1^n)+\sum_{j=2}^N\Delta x f(u_j^n) + \frac{\Delta x}{2}f(u_{N+1}^n)\right]\Bigg),
\end{eqnarray}
and
\begin{eqnarray}\label{eq:rightboundary}
 u_{N+1}^{n+1}\!=\!u_{N+1}^{n}\!\!\!&+&\!\!\!\Delta t\Bigg(-\frac{2}{\Delta x^2}\frac{u_{N+1}^n-u_N^n}{\sqrt{1+\left[\frac{u_{N+1}^n-u_N^n}{\Delta x}\right]^2}}+\lambda f(u_{N+1}^n)\nonumber \\
\!\!\!&-&\!\!\!\frac{\lambda}{L}\left[\frac{\Delta x}{2}f(u_{1}^n)+\sum_{j=2}^N\Delta x f(u_j^n) + \frac{\Delta x}{2}f(u_{N+1}^n)\right]\!\Bigg).
\end{eqnarray}

For the above numerical scheme to approximate~(\ref{problem}) reasonably it must also conserve mass. Therefore we have the following theorem.
\begin{thm}
The explicit numerical scheme contained in~(\ref{eq:interior}),~(\ref{eq:leftboundary}),~(\ref{eq:rightboundary}) conserves mass.
\end{thm}
\begin{proof}
We multiply~(\ref{eq:interior}) by $\Delta x$ and sum over $i=2,3,\cdots,N$ to obtain
\begin{eqnarray}\label{eq:sum}
\sum_{i=2}^Nu_i^{n+1}\Delta x \!&=&\! \sum_{i=2}^Nu_i^n\Delta x+\Delta t \Bigg(\sum_{i=2}^N\frac{1}{\Delta x}\left[\frac{u_{i+1}^n-u_i^n}{\sqrt{1+\left(\frac{u_{i+1}^n-u_i^n}{\Delta x}\right)^2}}-\frac{u_{i}^n-u_{i-1}^n}{\sqrt{1+\left(\frac{u_{i}^n-u_{i-1}^n}{\Delta x}\right)^2}}\right]\nonumber \\
&+&\lambda\sum_{i=2}^Nf(u_i^n)\Delta x -\sum_{i=2}^N\Delta x\frac{\lambda}{L}\left[\frac{\Delta x}{2}f(u_1^n)+\sum_{j=2}^N\Delta x f(u_j^n) + \frac{\Delta x}{2}f(u_{N+1}^n)\right]\Bigg)\nonumber \\
\!&=&\! \sum_{i=2}^Nu_i^n\Delta x+\Delta t \Bigg(\frac{1}{\Delta x}\left[\frac{u_{N+1}^n-u_N^n}{\sqrt{1+\left[\frac{u_{N+1}^n-u_N^n}{\Delta x}\right]^2}}-\frac{u_2^n-u_1^n}{\sqrt{1+\left[\frac{u_2^n-u_1^n}{\Delta x}\right]^2}}\right]\nonumber \\
&+&\lambda\sum_{i=2}^Nf(u_i^n)\Delta x-\sum_{i=2}^N\Delta x\frac{\lambda}{L}\left[\frac{\Delta x}{2}f(u_1^n)+\sum_{j=2}^N\Delta x f(u_j^n) + \frac{\Delta x}{2}f(u_{N+1}^n)\right]\Bigg) \nonumber \\
\end{eqnarray}
Now we multiply both of the boundary terms in~(\ref{eq:leftboundary}) and~(\ref{eq:rightboundary}) by $\frac{\Delta x}{2}$ and add the resulting equations to~(\ref{eq:sum}) to give
\begin{eqnarray}
 \left[\frac{u_1^{n+1}}{2}+\sum_{j=2}^Nu_i^{n+1}+\frac{u_{N+1}^{n+1}}{2}\right]\Delta x =\left[\frac{u_1^n}{2}+\sum_{j=2}^Nu_i^n+\frac{u_{N+1}^n}{2}\right]\Delta x\nonumber 
\end{eqnarray}
and so mass is conserved at the discrete level as required.
\end{proof}
\subsection*{Numerical experiments}
We present the results of some numerical simulations for the non-local mass-conserving equation~(\ref{problem}) using the explicit mass-conserving numerical scheme described in Section~\ref{sect:numerical scheme}. For $0<M<\frac{1}{\sqrt{5}}$ fixed, we choose both sub- and supercritical lengths $L$ (see Proposition~\ref{prop:mc direction of pitchfork}). In the subcritical case (Experiment~\ref{exp:subcritical}) we use AUTO to plot the bifurcation diagram for monotone classical solutions to~(\ref{eq:masscon}) and find the value $\lambda_1(M,L)$ of Theorem~\ref{thm:nomoreclassical} for these values of $M$ and $L$. Then we solve~(\ref{problem}) using~(\ref{eq:interior}),~(\ref{eq:leftboundary}),~(\ref{eq:rightboundary}) for $\lambda<\lambda_1(M,L)$ and for $\lambda>\lambda_1(M,L)$ with monotone initial data satisfying the mass constraint and present the initial data with the final equilibrium state in each case. In the supercritical case (Experiment~\ref{exp:supercritical}), we use AUTO to plot the bifurcation diagram for monotone classical solutions to~(\ref{eq:masscon}) and also for non-monotone classical solutions to~(\ref{eq:masscon}) which have two inflection points in $(0,L)$. We find $\lambda_1(M,L)$ and $\lambda_2(M,L)$ the values of $\lambda$ such that for all $\lambda>\lambda_i(M,L)$, there are no classical solutions to~(\ref{eq:masscon}) with $i$ inflection points in $(0,L)$ for $i=1,2$ respectively. We then run experiments solving~(\ref{problem}) in the cases $\lambda>\lambda_1(M,L)$ and $\lambda>\lambda_2(M,L)$ using~(\ref{eq:interior}),~(\ref{eq:leftboundary}),~(\ref{eq:rightboundary}) for particular initial data $u_0(x)$. \\

\begin{experiment}\label{exp:subcritical}
\textup{We fix $M=0.2$ so that $L^*=2.1856$ from~(\ref{Lstar}) and so, to have a subcritical $L$, we can take $L=1.7$ in this case. We use AUTO to plot the bifurcation diagram for monotone classical solutions to~(\ref{eq:masscon}) for these values of $L$ and $M$ in Figure~\ref{fig:bifdiag}. We have a subcritical bifurcation from the trivial solution $u(x)=M$ when $\lambda=\frac{\pi^2}{1.7^2f'(0.2)}=3.8808$ and according to AUTO, the bifurcation diagram for monotone classical solutions to the stationary problem stops when $\lambda=\lambda_1(0.2,1.7)\simeq 4.3032$ as in Figure~\ref{fig:bifdiag}.
\begin{figure}[H]
\begin{center}
\scalebox{0.665}{\psfrag{L}{\large $L^*$}
\psfrag{0}{\large $0.$}
\psfrag{1}{\large $1.$}
\psfrag{2}{\large $2.$}
\psfrag{3}{\large $3.$}
\psfrag{4}{\large $4.$}
\psfrag{5}{\large $5.$}
\psfrag{6}{\large $6.$}
\psfrag{r}{\large $u(0)$}
\psfrag{l}{\large $\lambda$}
\psfrag{-1}{\large $-1.$}
\psfrag{M}{\large $M$}
\includegraphics{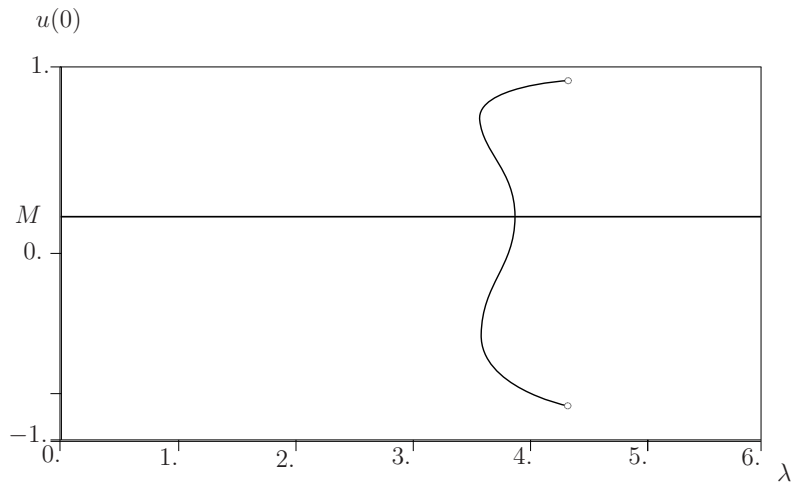}}
\end{center}
\caption{Bifurcation diagram for monotone classical solutions to~(\ref{eq:masscon}) with $M=0.2$, $L=1.7$ where $\lambda_1(M,L)\simeq4.3032$.} 
\label{fig:bifdiag}
\end{figure}
 We take the following initial data
\begin{eqnarray}\label{eq:mcinitialdata}
 u_0(x)=0.3-0.5\tanh\left(1000\left[\frac{x}{L}-0.4\right]\right),
\end{eqnarray}
which satisfies the mass constraint, i.e.
\begin{eqnarray}
  \frac{1}{L}\int_0^Lu_0(x)\,dx=M=0.2.\nonumber
\end{eqnarray}
Hence we present the equilibrium solutions to the time-dependent problem~(\ref{problem}), ~(\ref{eq:mcinitialdata}) in Figure~\ref{fig:M=0.2,L=1.7} for $\lambda=4<\lambda_1(M,L)$ (left) and $\lambda=5>\lambda_1(M,L)$ (right) obtained using the scheme in~(\ref{eq:interior}),~(\ref{eq:leftboundary}),~(\ref{eq:rightboundary}) with $N=500$.
\begin{figure}[H]
\begin{center}
\scalebox{0.38}{\includegraphics{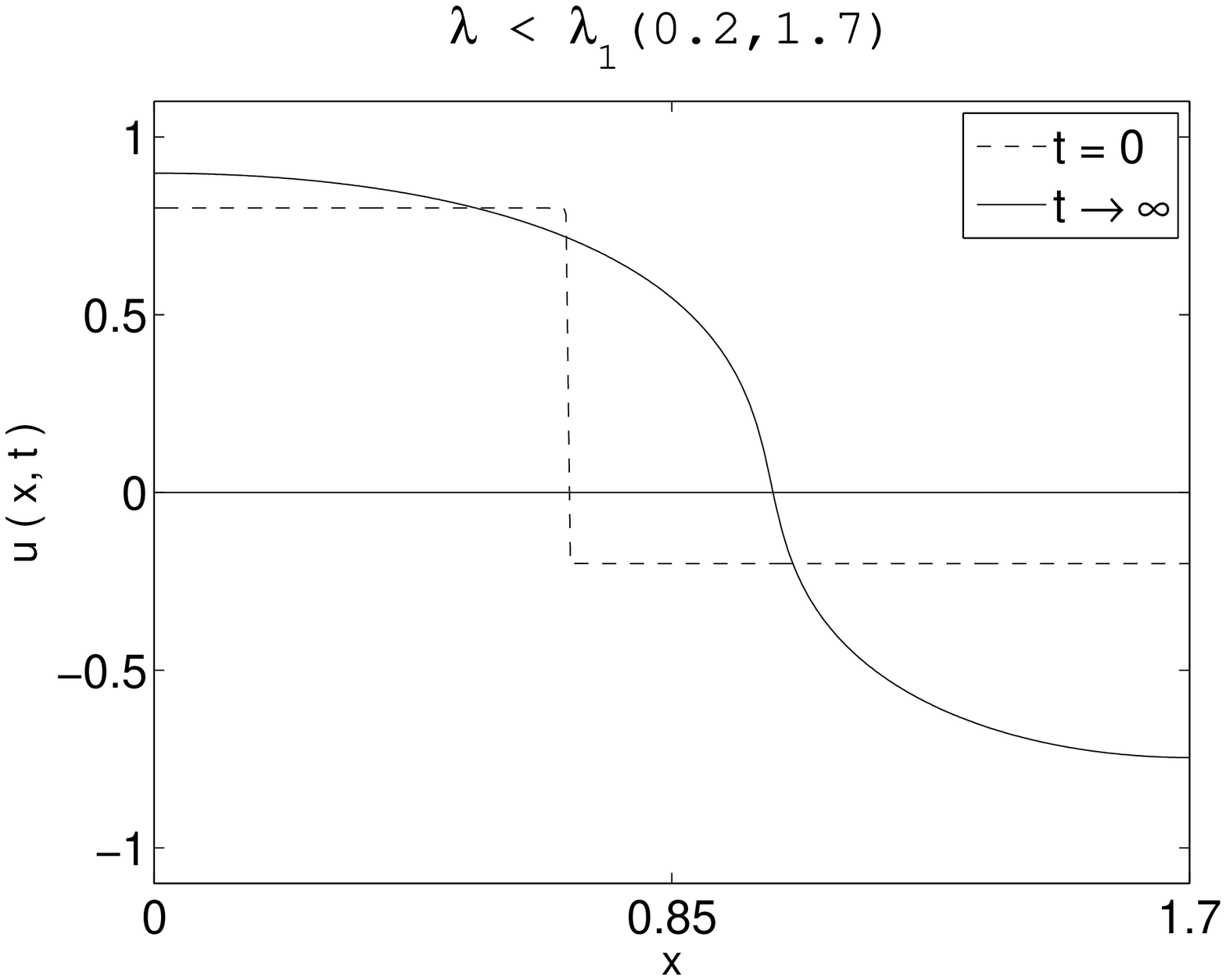}\includegraphics{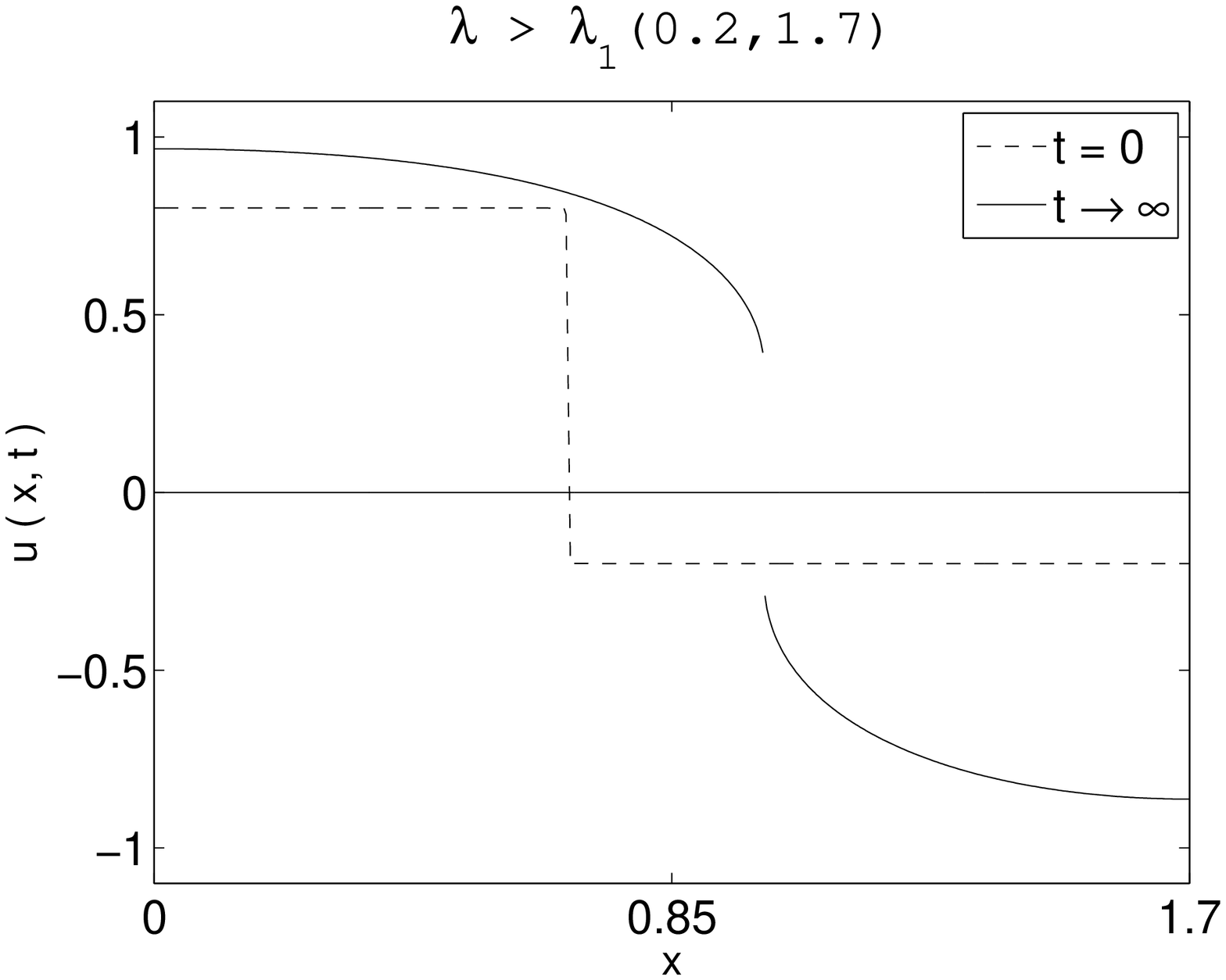}}
\end{center}
\caption{Initial data and equilibrium solutions to~(\ref{problem}),~(\ref{eq:mcinitialdata}) for $\lambda<\lambda_1(0.2,1.7)$ (left) and $\lambda>\lambda_1(0.2,1.7)$ (right).}
\label{fig:M=0.2,L=1.7}
\end{figure}}
\end{experiment}

\begin{experiment}\label{exp:supercritical}
\textup{Suppose we now consider the case $M=0.2$ with $L=2.5$. We have a supercritical bifurcation from the trivial solution $u(x)=M$ when $ \lambda=\frac{\pi^2}{2.5^2f'(0.2)}=1.7945$ and according to AUTO, the bifurcation diagram for monotone classical solutions to the stationary problem stops when $\lambda=\lambda_1(0.2,2.5)\simeq 4.0433$. There is also a subcritical bifurcation from the trivial solution when $ \lambda=\frac{4\pi^2}{2.5^2f'(0.2)}=7.1779$ and a curve of non-monotone classical solutions to~(\ref{eq:masscon}) which stops when $\lambda=\lambda_2(0.2,2.5)\simeq 4.9872$ just after it has reached a saddle-node at $\lambda=\lambda_{sn}\simeq4.9714$ as in Figure~\ref{fig:bifdiag2}.
\begin{figure}[H]
\begin{center}
\scalebox{0.69}{\psfrag{L}{\large $L^*$}
\psfrag{0}{\large $0.$}
\psfrag{1}{\large $1.$}
\psfrag{2}{\large $2.$}
\psfrag{4}{\large $4.$}
\psfrag{6}{\large $6.$}
\psfrag{8}{\large $8.$}
\psfrag{10}{\large $10.$}
\psfrag{r}{\large $u(0)$}
\psfrag{l}{\large $\lambda$}
\psfrag{-1}{\large $-1.$}
\psfrag{M}{\large $M$}
\includegraphics{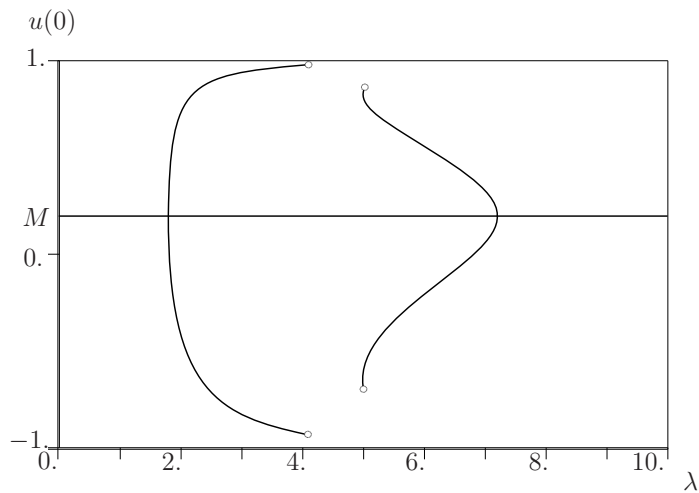}}
\end{center}
\caption{Bifurcation diagram for classical solutions to~(\ref{eq:masscon}) with at most two inflection points in $(0,L)$ where $M=0.2$, $L=2.5$, $\lambda_1(M,L)\simeq4.0433$, $\lambda_2(M,L)\simeq 4.9872$ and there is a saddle-node at $\lambda=\lambda_{sn}\simeq4.9714$.} 
\label{fig:bifdiag2}
\end{figure}}
\end{experiment}

 Suppose we solve~(\ref{problem}) for $L=2.5$, $M=0.2$ and fixed $\lambda>\lambda_1(M,L)$ with initial data given by
\begin{eqnarray}\label{interfaceic}
 u_0(x)=\beta +0.5\tanh\left(1000\left(\frac{x}{L}-\gamma\right)\right),
\end{eqnarray}
and we vary the position $x_0=\gamma L$ of the interface in~(\ref{interfaceic}) but ensure that 
\begin{eqnarray}\label{eq:massconstraint}
 \frac{1}{L}\int_0^Lu_0(x)=M,
\end{eqnarray}
still holds by changing $\beta$ accordingly. The results of taking $L=2.5$, $M=0.2$ with $\lambda=8>\lambda_1(0.2,2.5)\simeq 4.0433$ and solving~(\ref{problem}) for various values of $\gamma$ (and $\beta$) are plotted in Figure~\ref{fig:overlay} where one sees that modifying the initial data while ensuring that~(\ref{eq:massconstraint}) holds has an effect on the equilibrium state to which the solution converges as $t\rightarrow\infty$ which is not the case for $\lambda<\lambda_1(M,L)$. This suggests that the discontinuous equilibria for~(\ref{eq:massconrose}) existing for $\lambda>\lambda_1(M,L)$ are \textbf{normally} stable in the sense of \cite{PSZ}.
\begin{figure}[H]
\begin{center}
\scalebox{0.35}{\includegraphics{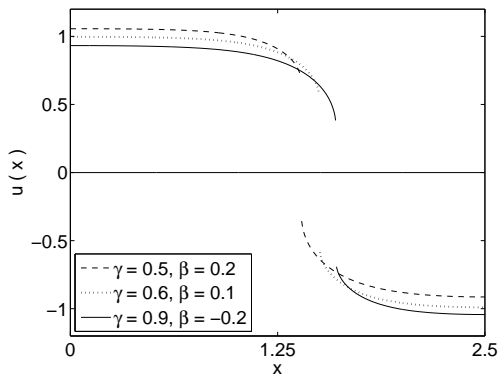}}
\end{center}
\caption{Equilibrium solutions to~(\ref{problem}),~(\ref{interfaceic}) with $(\gamma,\beta)=(0.5,0.2),\,\,(0.6,0.1),\,\,(0.9,-0.2)$  solved using~(\ref{eq:interior}),~(\ref{eq:leftboundary}),~(\ref{eq:rightboundary}) with $N=500$.}
\label{fig:overlay}
\end{figure}
 Finally we take $\lambda>\lambda_2(0.2,2.5)$ with the following non-monotone initial data we define piecewise as follows
 \begin{equation}\label{eq:nonmonotone}
   u_0(x) = \left\{ 
\begin{array}{ll}
0.32-0.6\tanh\left(1000\left[\frac{x}{L}-0.2\right]\right) \hspace{0.3cm} 0&\le x <\frac{L}{2}\\
0.32+0.6\tanh\left(1000\left[\frac{x}{L}-0.8\right]\right) \hspace{0.3cm} \frac{L}{2}&\le x <L
\end{array} \right.,
\end{equation}
which satisfies the mass constraint~(\ref{eq:massconstraint}). In Figure~\ref{fig:nonmonotone} we show this initial data and the final equilibrium state to which the solution converges as $t\rightarrow\infty$ for $\lambda=8>\lambda_2(0.2,2.5)$. 
\begin{figure}[H]
\begin{center}
\scalebox{0.35}{\includegraphics{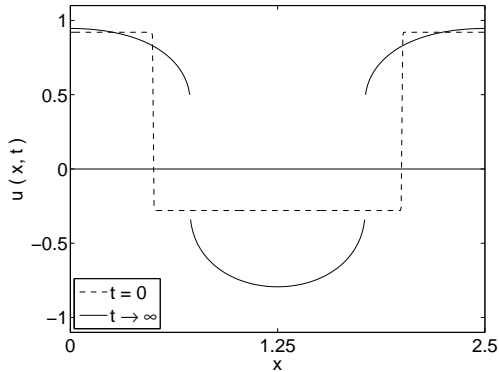}}
\end{center}
\caption{Non-monotone non-classical equilibrium solution to~(\ref{problem}),~(\ref{eq:nonmonotone}) for $\lambda>\lambda_2(0.2,2.5)\simeq 4.9872$ solved using~(\ref{eq:interior}),~(\ref{eq:leftboundary}),~(\ref{eq:rightboundary}) with $N=500$.} 
\label{fig:nonmonotone}
\end{figure}

\section{Conclusions}

We have introduced a non-local mass-conserving quasilinear equation~(\ref{eq:massconrose}) that can be viewed as a one-dimensional quasilinear version of the Rubinstein-Sternberg equation from \cite{RS} and we have considered its associated stationary problem. We have shown in Proposition~\ref{prop:mc direction of pitchfork} that unlike the situation for the bifurcation diagrams of stationary solutions to the one-dimensional Cahn-Hilliard equation studied in \cite{EFG}, the bifurcation behaviour for classical stationary solutions to~(\ref{eq:massconrose}) as $\lambda$ is varied depends on the length $L$ of the space interval $\Omega$ as well as on the average mass $M$ of a solution. We have also proved (Theorem~\ref{thm:nomoreclassical}) that for each $n\in\mathbb{N}$ there is a value of $\lambda=\lambda_n(M,L)$ such that for any $\lambda>\lambda_n(M,L)$ there cannot exist classical stationary solutions to~(\ref{eq:massconrose}) with $n$ points of inflection in $(0,L)$. We have presented some numerical results using AUTO which illustrate Proposition~\ref{prop:mc direction of pitchfork} and Theorem~\ref{thm:nomoreclassical} and we derived a mass-conserving numerical scheme for~(\ref{eq:massconrose}). There is numerical evidence (see Figure~\ref{fig:overlay}) to suggest that for $\lambda>\lambda_1(M,L)$, there exist a continuum of \textbf{normally} stable (in the sense of \cite{PSZ}) non-classical monotone stationary solutions to~(\ref{eq:massconrose}) with a discontinuity at some $x_0\in (0,L)$. We have also presented numerical evidence which suggests that there are also {\em stable} non-monotone non-classical stationary solutions to~(\ref{eq:massconrose}) for $\lambda$ large enough (see Figure~\ref{fig:nonmonotone}).\\

An understanding of exactly what happens in the ``metastable'' regime $\left(\frac{1}{\sqrt{3}}<|M|<1\right)$ is required.  We were not able to perform a two parameter continuation in $\lambda$ and in $M$ of the saddle-nodes arising in the $\frac{1}{\sqrt{5}}<|M|<\frac{1}{\sqrt{3}}$ parameter regime since for $M$ large enough, the classical solutions to~(\ref{eq:masscon}) stop existing in the bifurcation diagrams before the saddle-nodes are reached. One can show however through an analysis of the phase plane associated with the ancillary problem for~(\ref{eq:masscon}) that for $|M|\ge1$, there are no non-trivial solutions to~(\ref{eq:masscon}).

We also note that no existence or stabilisation theory for (\ref{eq:massconrose}) or (\ref{eq:CHRose}) is available. 

We hope that the class of models introduced in this paper will be of interest to material scientists and that its applicability to real materials will be tested as has been done for more classical models in, for example, \cite{HS}.

\section*{Appendix}

In this appendix, we present the details of the Liapunov-Schmidt Reduction we discussed in Section~\ref{sect:statprob}. For the operator $G:D(G)\rightarrow H$ given in~(\ref{eq:G}), let $ S=(dG)_{0,\lambda_k,M}:D(G)\rightarrow H$, with kernel $\mathcal{K}$ and range $\mathcal{R}$ given by 
\begin{equation}\label{eq:range(S)}
\mathcal{K}=\left\{\textup{span}\left[\cos\left(\frac{k\pi
x}{L}\right)\right]\right\}
\hbox{ and } \mathcal{R}=\left\{w\in H:\int_0^Lw(x)\cos\left(\frac{k\pi x}{L}\right)\,dx=0\right\},
\end{equation}
respectively, and let $E:H\rightarrow\mathcal{R}$ denote the projection of $H$ onto $\mathcal{R}$. The linearisation operator $S$ for equation~(\ref{eq:masscon}) at a bifurcation point $(\lambda_k,0)$ is a self-adjoint elliptic operator and therefore Fredholm of index zero (see \cite[Appendix 4]{GS}) and so, following the steps in \cite[p.29]{GS}, the spaces $D(G)$ and $H$ are decomposed as 
\[
 D(G)=\mathcal{K}\oplus\mathcal{K}^\perp,\hspace{0.5in}H=\mathcal{K}\oplus\mathcal{K}^\perp,\nonumber
\]
since $\mathcal{K}=\mathcal{R}^\perp$ and $\mathcal{K}^\perp=\mathcal{R}$. The coordinates chosen in the Liapunov-Schmidt reduction are then \begin{eqnarray}
 v_k^*=2\cos\left(\frac{k\pi x}{L}\right),\hspace{0.1in} v_k=\cos\left(\frac{k\pi x}{L}\right).\nonumber
\end{eqnarray}

From \cite[Chapter 1]{GS}, the derivatives of the reduced function $h(\lambda,y)$ evaluated at $\lambda=\lambda_k$, $y=0$ are given by 
\begin{eqnarray}\label{eq:derivatives of h}
 h_y\!&=&\!\left\langle v_k^*,dG(v_k)\right\rangle=0\nonumber \\ 
 h_{yy}\!&=&\!\left\langle v_k^*,d^2G(v_k,v_k)\right\rangle\nonumber \\
 h_{yyy}\!&=&\!\left\langle v_k^*,d^3G(v_k,v_k,v_k)-3d^2G(v_k,S^{-1}E[d^2G(v_k,v_k)])\right\rangle \\
 h_{\lambda}\!&=&\!\left\langle v_k^*,G_{\lambda}\right\rangle\nonumber \\
 h_{\lambda y}\!&=&\!\left\langle v_k^*,dG_{\lambda}(v_k)-d^2G(v_k,S^{-1}EG_{\lambda})\right\rangle, \nonumber 
\end{eqnarray}
where $\langle\cdot,\cdot\rangle$ denotes the $L^2$-inner product on $[0,L]$. \\

We now proceed to calculate the derivatives of the reduced function $h(\lambda,y)$ in order to determine locally the direction of the pitchfork bifurcations from the trivial solution $u(x)=M$ of~(\ref{eq:masscon}) for a given $L$ and a given $M$. This is complicated by the fact that the operator $G$ in~(\ref{eq:G}) is not odd and so terms involving $S^{-1}$ in~(\ref{eq:derivatives of h}) will not necessarily vanish. In order to invert $S$ we will need to solve an ordinary differential equation; see~(\ref{eq:ODE}). \\

We have
\begin{eqnarray}
 (d^2G)_{0,\lambda_k,M}(w_1,w_2)\!&=&\!\frac{\partial^2}{\partial t_1\partial t_2}G(0+t_1w_1+t_2w_2,\lambda_k,M)\nonumber \\
 \!&=&\!\lambda_k f''(M)w_1w_2-\frac{\lambda_k}{L}\int_0^Lf''(M)w_1(x)w_2(x)\,dx,\nonumber
\end{eqnarray}
and so
\begin{eqnarray}
  (d^2G)_{0,\lambda_k,M}(v_k,v_k)\!&=&\!\lambda_k f''(M)\cos^2\left(\frac{k\pi x}{L}\right)-\frac{\lambda_k}{L}\int_0^Lf''(M)\cos^2\left(\frac{k\pi x}{L}\right)\,dx\nonumber \\
  \!&=&\!\lambda_k f''(M)\cos^2\left(\frac{k\pi x}{L}\right)-\lambda_k\frac{f''(M)}{2}.\nonumber
\end{eqnarray}
Hence by~(\ref{eq:derivatives of h})
\begin{eqnarray}
 h_{yy}\!&=&\!\left\langle v_k^*,d^2G(v_k,v_k)\right\rangle\nonumber \\
 \!&=&\!\int_0^L2\cos\left(\frac{k\pi x}{L}\right)\left[\lambda_kf''(M)\cos^2\left(\frac{k\pi x}{L}\right)-\lambda_k\frac{f''(M)}{2}\right]\,d x\nonumber\\
 \!&=&\!0.\nonumber
\end{eqnarray}
Set
\begin{eqnarray} 
(d^2G)_{0,\lambda_k,M}(v_k,v_k)=\lambda_kf''(M)\cos^2\left(\frac{k\pi x}{L}\right)-\lambda_k\frac{f''(M)}{2}=p(x).\nonumber
\end{eqnarray}
so that 
\begin{eqnarray}
 p'(0)=p'(L)=0,\nonumber
\end{eqnarray}
and
\begin{eqnarray}
 \frac{1}{L}\int_0^Lp(x)\,dx\!&=&\!0,\nonumber 
\end{eqnarray}
therefore $(d^2G)_{0,\lambda_k,M}(v_k,v_k)\in H$. \\

However, 
\begin{eqnarray}
 \int_0^L(d^2G)_{0,\lambda_k,M}\hspace{-0.35cm}& &\hspace{-0.35cm}(v_k,v_k)\cos\left(\frac{k\pi x}{L}\right)\,dx\nonumber \\
 \!&=&\!\int_0^L\left[\lambda_kf''(M)\cos^3\left(\frac{k\pi x}{L}\right)-\lambda_k\frac{f''(M)}{2}\cos\left(\frac{k\pi x}{L}\right)\right]\,dx\nonumber \\
\!&=&\!0,\nonumber
\end{eqnarray}
so that $(d^2G)_{0,\lambda_k,M}(v_k,v_k)\in \mathcal{R}$ by~(\ref{eq:range(S)}). Therefore $(d^2G)_{0,\lambda_k,M}(v_k,v_k)\in H$ trivially decomposes as
\begin{eqnarray}
 (d^2G)_{0,\lambda_k,M}(v_k,v_k)=(d^2G)_{0,\lambda_k,M}(v_k,v_k)+0,\nonumber
\end{eqnarray}
where $(d^2G)_{0,\lambda_k,M}(v_k,v_k)\in \mathcal{R}$ and of course $0\in \mathcal{R}^\perp$. Thus, since $E:H\rightarrow \mathcal{R}$ is the projection of $H$ onto the range of $S$, we have
\begin{eqnarray}
E[(d^2G)_{0,\lambda_k,M}(v_k,v_k)]=(d^2G)_{0,\lambda_k,M}(v_k,v_k),\nonumber
\end{eqnarray}
and we consider
\begin{eqnarray}
 S^{-1}E[(d^2G)_{0,\lambda_k,M}(v_k,v_k)]\!&=&\!S^{-1}(d^2G)_{0,\lambda_k,M}(v_k,v_k)=l(x)\nonumber \\
 \Rightarrow(d^2G)_{0,\lambda_k,M}(v_k,v_k)\!&=&\!Sl(x).\nonumber
\end{eqnarray}
Thus the second order ordinary differential equation that we need to solve for $l(x)$ in order to obtain a formula for $S^{-1}E[(d^2G)_{0,\lambda_k,M}(v_k,v_k)]$ is
\begin{eqnarray}\label{eq:ODE}
 l''(x)+\lambda_kf'(M)l(x)=\lambda_kf''(M)\cos^2\left(\frac{k\pi x}{L}\right)-\lambda_k\frac{f''(M)}{2},
\end{eqnarray}
which has solution 
\begin{eqnarray}
 l(x)=\cos\left(\frac{k\pi x}{L}\right)-\frac{1}{6}\frac{f''(M)}{f'(M)}\cos\left(\frac{2k\pi x}{L}\right),\nonumber
\end{eqnarray}
so that
\begin{eqnarray}
S^{-1}E[(d^2G)_{0,\lambda_k,M}(v_k,v_k)]=\cos\left(\frac{k\pi x}{L}\right)-\frac{1}{6}\frac{f''(M)}{f'(M)}\cos\left(\frac{2k\pi x}{L}\right).\nonumber
\end{eqnarray}
Hence we compute
\begin{eqnarray}
 d^2G(v_k,S\hspace{-0.35cm}& &\hspace{-0.35cm}^{-1}E[d^2G(v_k,v_k)])\nonumber \\
 \!&=&\!d^2G\left(\cos\left(\frac{k\pi x}{L}\right),\cos\left(\frac{k\pi x}{L}\right)-\frac{1}{6}\frac{f''(M)}{f'(M)}\cos\left(\frac{2k\pi x}{L}\right)\right)\nonumber \\
 \!&=&\!\lambda_kf''(M)\left[\cos^2\left(\frac{k\pi x}{L}\right)-\frac{1}{6}\frac{f''(M)}{f'(M)}\cos\left(\frac{2k\pi x}{L}\right)\cos\left(\frac{k\pi x}{L}\right)\right] \nonumber \\
 \hspace{-0.35cm}& &\hspace{-0.35cm}-\lambda_k\frac{1}{L}\int_0^Lf''(M)\left[\cos^2\left(\frac{k\pi x}{L}\right)-\frac{1}{6}\frac{f''(M)}{f'(M)}\cos\left(\frac{2k\pi x}{L}\right)\cos\left(\frac{k\pi x}{L}\right)\right]\,dx\nonumber \\
 \!&=&\!\lambda_kf''(M)\left[\cos^2\left(\frac{k\pi x}{L}\right)-\frac{1}{6}\frac{f''(M)}{f'(M)}\cos\left(\frac{2k\pi x}{L}\right)\cos\left(\frac{k\pi x}{L}\right)\right]-\lambda_k\frac{f''(M)}{2}, \nonumber
\end{eqnarray}
and so we have
\begin{eqnarray}\label{eq:1stpart}
\langle v_k^*,3d^2G(v_k,S\hspace{-0.35cm}& &\hspace{-0.35cm}^{-1}E[d^2G(v_k,v_k)])\rangle\nonumber \\
\!&=&\!\int_0^L6\lambda_kf''(M)\left[\cos^3\left(\frac{k\pi x}{L}\right)-\frac{1}{6}\frac{f''(M)}{f'(M)}\cos\left(\frac{2k\pi x}{L}\right)\cos^2\left(\frac{k\pi x}{L}\right)\right]\,dx\nonumber \\
\hspace{-0.35cm}& &\hspace{-0.35cm}-\int_0^L3\lambda_kf''(M)\cos\left(\frac{k\pi x}{L}\right)\,dx\nonumber \\
\!&=&\!-\frac{\lambda_k L}{4}\frac{[f''(M)]^2}{f'(M)}=-\frac{k^2\pi^2}{4L}\frac{[f''(M)]^2}{[f'(M)]^2}.
\end{eqnarray}
In addition to this,
\begin{eqnarray}
 (d^3G)_{0,\lambda_k,M}(w_1,w_2,w_3)\!&=&\!\frac{\partial^3}{\partial t_1\partial t_2\partial t_3}G(0+t_1w_1+t_2w_2+t_3w_3,\lambda_k,M)|_{t_1=t_2=t_3=0}\nonumber \\
 \!&=&\!-3[w_3''w_1'w_2'+w_2''w_1'w_3'+w_1''w_2'w_3']+\lambda_kf'''(M)w_1w_2w_3\nonumber \\
\hspace{-0.35cm}& &\hspace{-0.35cm}-\lambda_kf'''(M)\frac{1}{L}\int_0^Lw_1(x)w_2(x)w_3(x)\,dx,\nonumber
\end{eqnarray}
so that
\begin{eqnarray}
 (d^3G)_{0,\lambda_k,M}(v_k,v_k,v_k)\hspace{-0.15cm}&-&\hspace{-0.15cm}9[v_k''(v_k')^2]+\lambda_kf'''(M)v_k^3-\lambda_kf'''(M)\frac{1}{L}\int_0^Lv_k^3(x)\,dx\nonumber \\
\!&=&\!\frac{9k^4\pi^4}{L^4}\cos\left(\frac{k\pi x}{L}\right)\sin^2\left(\frac{k\pi x}{L}\right)+\lambda_kf'''(M)\cos^3\left(\frac{k\pi x}{L}\right),\nonumber  
\end{eqnarray}
and
\begin{eqnarray}\label{eq:2ndpart}
 \left\langle v_k^*,d^3G(v_k,v_k,v_k)\right\rangle \!&=&\!\int_0^L\left[\frac{18 k^4\pi^4}{L^4}\cos^2\left(\frac{k\pi x}{L}\right)\sin^2\left(\frac{k\pi x}{L}\right)+2\lambda_kf'''(M)\cos^4\left(\frac{\pi x}{L}\right)\right]\,dx\nonumber \\
 \!&=&\!\frac{3k^2\pi^2}{4L^3}\left(3k^2\pi^2+L^2\frac{f'''(M)}{f'(M)}\right).
\end{eqnarray}
Therefore from~(\ref{eq:derivatives of h}),~(\ref{eq:1stpart}) and~(\ref{eq:2ndpart}) we obtain
\begin{eqnarray}\label{eq:hyyy}
 h_{yyy}\!&=&\!\left\langle v_k^*,d^3G(v_k,v_k,v_k)-3d^2G(v_k,S^{-1}E[d^2G(v_k,v_k)])\right\rangle\nonumber \\
 \!&=&\!\frac{3k^2\pi^2}{4L^3[f'(M)]^2}\left(3k^2\pi^2[f'(M)]^2+L^2f'''(M)f'(M)+\frac{L^2}{3}[f''(M)]^2\right).
\end{eqnarray}
Also, 
\begin{eqnarray}
 G_\lambda(v,\lambda,M)=f(v+M)-\frac{1}{L}\int_0^Lf(v(x)+M)\,dx,\nonumber
\end{eqnarray}
so that  $G_\lambda(0,\lambda_k,M)=0$ which implies that
\begin{eqnarray}\label{eq:2ndpart0}
 (d^2G)_{0,\lambda_k,M}(v_k,S^{-1}E[G_\lambda(0,\lambda_k,M)])=0,
\end{eqnarray}
while
\begin{eqnarray}\label{eq:1stpartn0}
 (dG_\lambda)_{0,\lambda_k,M}\cdot w\!&=&\!f'(M)w-\frac{1}{L}\int_0^Lf'(M)w(x)\,dx\nonumber \\
\!&=&\!f'(M)w,
\end{eqnarray}
for any $w\in D(G)$. Therefore, from~(\ref{eq:derivatives of h}),~(\ref{eq:2ndpart0}) and~(\ref{eq:1stpartn0}) we have
\begin{eqnarray}\label{eq:hlambday}
  h_{\lambda y}\!&=&\!\left\langle v_k^*,dG_{\lambda}(v_k)-d^2G(v_k,S^{-1}EG_\lambda)\right\rangle \nonumber \\
  \!&=&\!\int_0^L2\cos\left(\frac{k\pi x}{L}\right)f'(M)\cos\left(\frac{k\pi x}{L}\right)\,dx\nonumber \\
  \!&=&\!Lf'(M).
\end{eqnarray}

\end{document}